\documentclass[sn-mathphys,Numbered]{sn-jnl} 

\usepackage{graphicx} 
\usepackage{multirow} 
\usepackage{amsmath,amssymb,amsfonts} 
\usepackage{amsthm} 
\usepackage{mathrsfs} 
\usepackage[title]{appendix} 
\usepackage{xcolor} 
\usepackage{float}
\usepackage{textcomp} 
\usepackage{manyfoot} 
\usepackage{algorithm} 
\usepackage{algorithmicx} 
\usepackage{algpseudocode} 
\usepackage{listings}
\usepackage{array}
 \usepackage{caption}
 \usepackage{subcaption}
 \usepackage[utf8]{inputenc}
\usepackage{booktabs}
\usepackage{amsfonts}
\usepackage{siunitx}
\usepackage{longtable}
\usepackage{eucal}
\usepackage{multirow}
\usepackage{enumerate}
\usepackage{lscape}
\usepackage{framed}
\usepackage{anyfontsize}
\usepackage{pifont} 
\usepackage{marvosym}
\usepackage{lineno}
\usepackage{physics}
\usepackage{mathtools}

\newtheorem{theorem}{Theorem}[section]
\newtheorem{proposition}{Proposition}[section] 
\newtheorem{lemma}{Lemma}[section] 
  
\newtheorem{example}{Example}[section] 
\newtheorem{remark}{Remark}[section] 
\newtheorem{assumption}{Assumption}[section]
\newtheorem{definition}{Definition}[section]

\DeclareMathOperator*{\argmin}{argmin}
\DeclareMathOperator*{\argmax}{argmax}

\usepackage{enumitem}
\begin{document}
\title[Article Title]{Nonlinear Conjugate Gradient Methods for Optimization of Set-Valued Mappings of Finite Cardinality}

\author[1]{\fnm{Debdas} \sur{Ghosh}}\email{debdas.mat@iitbhu.ac.in}

\author[1]{\fnm{Ravi} \sur{Raushan}}\email{raviraushan.rs.mat21@itbhu.ac.in}

\author*[2]{\fnm{Zai-Yun} \sur{Peng}}\email{pengzaiyun@126.com}

\author[3,4]{\fnm{Jen-Chih} \sur{Yao}}\email{yaojc@math.nsysu.edu.tw}

\affil[1]{\orgdiv{Department of Mathematical Sciences}, \orgname{Indian Institute of Technology (BHU)}, \orgaddress{\city{Varanasi}, \postcode{221005}, \state{Uttar Pradesh}, \country{India}}}

\affil[1]{\orgdiv{College of Mathematics and Statistics}, \orgname{Chongqing JiaoTong University}, \orgaddress{\city{Chongqing}, \postcode{400074},  \country{P.R. China}}}

\affil[3]{\orgdiv{Center for General Education}, \orgname{China Medical University}, \orgaddress{\city{Taichung}, \country{Taiwan}}}

\affil[4]{\orgname{Academy of
Romanian Scientists}, \orgaddress{\city{Buchares}, \postcode{50044}, \country{Romania}}}

\abstract{This article presents nonlinear conjugate gradient methods for finding local weakly minimal points of set-valued optimization problems under a lower set less ordering relation. The set-valued objective function of the optimization problem under consideration is defined by finitely many continuously differentiable vector-valued functions. For such optimization problems, at first, we propose a general scheme for nonlinear conjugate gradient methods and then introduce Dai-Yuan, Polak-Ribi{\`e}re-Polyak, and Hestenes-Stiefel conjugate gradient parameters for set-valued functions. Toward deriving the general scheme, we introduce a condition of sufficient decrease and Wolfe line searches for set-valued functions. For a given sequence of descent directions of a set-valued function, it is found that if the proposed standard Wolfe line search technique is employed, then the generated sequence of iterates for set optimization follows a Zoutendijk-like condition. With the help of the derived Zoutendijk-like condition, we report that all the proposed nonlinear conjugate gradient schemes are globally convergent under usual assumptions. It is important to note that the ordering cone used in the entire study is not restricted to be finitely generated, and no regularity assumption on the solution set of the problem is required for any of the reported convergence analyses. Finally, we demonstrate the performance of the proposed methods through numerical experiments. In the numerical experiments, we demonstrate the effectiveness of the proposed methods not only on the commonly used test instances for set optimization but also on a few newly introduced problems under general ordering cones that are neither nonnegative hyper-octant nor finitely generated. \\

\noindent{\textbf{Keywords} Set optimization, Vector optimization, Conjugate gradient method, Wolfe conditions.}\\ 

\smallskip

 \noindent{\textbf{Mathematics Subject Classifications}
49J53, 90C29, 90C47.
 }}

\maketitle

\section{Introduction}

Set optimization is a class of mathematical problems concerned with minimizing set-valued mappings between two vector spaces, where the image space is partially ordered by a specified closed, convex, and pointed cone. Two primary approaches exist for defining solutions to these problems: the vector approach and the set approach. This paper centers on the latter.
In the set approach, a preorder is defined on the power set of the image space, and minimal solutions of the set-valued problem are determined accordingly. The foundational research in this field began with the works of Young \cite{young1931algebra}, Nishnianidze \cite{Nishnianidze1984}, and Kuroiwa \cite{kuroiwa1998natural,kuroiwa2001set}, who first introduced set relations to establish a preorder. Kuroiwa \cite{kuroiwa1997some} was the pioneer in applying the set approach to set optimization problems. Since then, this area of research has expanded significantly due to its applications in finance, optimization under uncertainty, game theory, and socioeconomics. For a comprehensive overview of the field, we refer to the monograph \cite{khan2016set}.  \\

The focus of this work is on the development of efficient numerical techniques for solving set optimization problems. The existing approaches to solve set optimization problems can be broadly categorized into the following five groups. \\ 
\begin{itemize}
    \item \emph{Derivative-Free Methods}  \cite{jahn2015derivative,jahn2018derivative,kobis2016treatment}:
    The methods in this group are descent methods that employ a derivative-free strategy. These algorithms are designed for unconstrained problems and do not assume any specific structure for the set-valued objective mapping. \\ 
    % A notable feature of the algorithm described in this class is its consideration of multiple descent directions at each iteration rather than selecting just one. Consequently, the method generates a tree structure with the initial point as the root and the potential solutions as leaves.

\item \emph{Sorting Type Algorithms}  \cite{günther2019strictly,gunther2019computing,kobis2016generalized,kobis2018numerical}: This class of methods is designed to tackle set optimization problems with a finite feasible set. As a result, they rely on direct comparisons of the images produced by the set-valued objective mapping. \\

\item \emph{Algorithms Based on Scalarization}  \cite{ehrgott2014minmax,eichfelder2020algorithmic,ide2014concepts,ide2014relationship,jiang2019generalized,schmidt2019min}: This category of methods adopt a scalarization approach and are intended for problems where the set-valued objective mapping has a structure derived from the robust counterpart of a vector optimization problem under uncertainty. Subsequently, various scalarization techniques are utilized to solve the set optimization problem. \\

\item \emph{Branch and Bound Approach} \cite{eichfelder2020algorithmic}: In the study by Eichfelder et al. \cite{eichfelder2020algorithmic}, an algorithm is introduced to tackle uncertain vector optimization problems, with the assumption that uncertainty arises solely from the decision variable. They proposed a branch-and-bound technique designed to identify a box covering for the solution set.  \\

\item \emph{First-Order Solution Methods}: Bouza et al. \cite{bouza2021steepest} and Kumar et al. \cite{kumar2024nonlinear} reported algorithms to tackle unconstrained set optimization problems where the set-valued objective mapping is defined by a finite number of continuously differentiable functions. While derivative-free methods can be used to solve such problems, they face similar limitations to those encountered in scalar optimization. Notably, derivative-free methods are expected to perform slower than first-order methods as they do not use the available first-order information of the objective function.  \\ 
\end{itemize}

In this paper, in the category of first-order methods, we introduce a general conjugate gradient algorithm for the set optimization problems taken in \cite{bouza2021steepest,kumar2024nonlinear}. The motivations for this work are  as follows: \\

\begin{itemize}
    \item The considered type of set optimization problems have significant applications in optimization under uncertainty. Indeed, set optimization problems of this nature emerge when determining robust solutions for vector optimization problems under uncertainty, especially when the uncertainty set is finite (see \cite{ide2014relationship}). Moreover, solving problems with a finite uncertainty set is essential for tackling the more general case with an infinite uncertainty set. This is demonstrated by the cutting plane strategy in \cite{mutapcic2009cutting} and the reduction techniques in \cite[Proposition 2.1]{ben1998robust} and \cite[Theorem 5.9]{ehrgott2014minmax}. \\

    \item In \cite{bouza2021steepest}, a regularity condition on the solution points is used to prove the convergence of the proposed algorithm. A regularity condition on the solution set, which is not apriori known, is restrictive. Thus, in this paper, we aim to devise a first-order technique in this study whose convergence does not depend on any regularity condition on the solution set of the problem. \\

    \item In \cite{lucambio2018nonlinear}, Prudente et al. extended the conventional nonlinear conjugate gradient methods for vector optimization problems. Because of \cite[Proposition 3.2]{lucambio2018nonlinear}, the methods in \cite{lucambio2018nonlinear}  are applicable only if the ordering cone is finitely generated. Although considerable work has been done on the direction of conjugate gradient methods for vector optimization  \cite{lucambio2018nonlinear,gonccalves2020extension,gonccalves2022study,he2023spectral,yahaya2024efficient}, to the best of our knowledge none of the existing methods is applicable for nonfinitely generated ordering cones. Thus, there is a need to derive a general nonlinear conjugate gradient method for the set optimization problems under consideration whose convergence is not dependent on the existence of a finite generator of the ordering cone. \\     
\end{itemize}

With these motivations, we propose a general nonlinear conjugate gradient method for set optimization problems without assuming that the ordering cone is finitely generated. We also establish the convergence of the proposed method without taking any regularity assumption on the solution set of the problem under consideration. \\

The outline of the rest of this paper is as follows. In Section \ref{sec:2}, we introduce notations, concepts, and results that are used throughout the paper. Section \ref{sec_Ana_wofe} analyzes the standard and the strong Wolfe line searches for set-valued functions and identifies a Zoutendijk-like condition. In Section \ref{section4}, we propose a general scheme for nonlinear conjugate gradient methods and analyze its global convergence. Three particular conjugate gradient parameters are provided in Section \ref{three_special_methods}. We illustrate the numerical performance of the proposed methods on various problems in Section \ref{Numerical Experiments}; a comparison of the proposed methods with the existing conjugate gradient methods for set optimization is also provided. Finally, in Section \ref{conclusion}, we conclude the entire study and propose ideas for further research.

\section{Preliminaries and terminologies}\label{sec:2}
In this section, we provide basic terminologies, notations, and results that are used throughout the paper. The positive orthant and nonnegative orthant of $\mathbb{R}^m$ are denoted by $\mathbb{R}^{m++}$ and $\mathbb{R}^{m+}$, respectively. For a given $p\in \mathbb{N}$, we use the notation $[p]$ to represent the set $\{1,2,\ldots,p\}$. The notation $\mathscr{P}(\mathbb{R}^{m})$ represents the class of all nonempty subsets of $\mathbb{R}^{m}$. For a set $A\in \mathscr{P}(\mathbb{R}^{m})$, $\mathrm{conv}(A)$ and $\mathrm{int}(A)$ denote its convex hull and interior, respectively. \\ 
%%%%%CONE%%%%%%

Throughout the paper, let $K\in \mathscr{P}(\mathbb{R}^{m})$ be a cone that is closed, convex, solid ($\mathrm{int}(K)\neq 
 \emptyset$), and pointed ($K\cap (-K)=\{0\}$). The partial order in $\mathbb{R}^{m}$ induced by $K$ is denoted by $\preceq_{K}$, and defined as 
\[y\preceq_{K} z \Longleftrightarrow z-y\in K,\]\\
and the strict order relation $\prec_{K}$ in $\mathbb{R}^{m}$ induced by $\mathrm{int} (K)$ is defined as
$$y\prec_{K} z \Longleftrightarrow z-y\in \mathrm{int}(K).$$
The positive polar cone of $K$ is the set $$K^{*}:=\{w\in \mathbb{R}^{m} \mid  w^{\top}y\geq 0 \text{ for all } y\in K\}.$$ Since $K$ is closed and convex,  we have $K=K^{**}$, $$-K=\{y\in \mathbb{R}^{m} \mid y^{\top}w\leq 0 \text{ for all } w\in K^{*}\},$$ and  $-\mathrm{int}(K)=\{y\in \mathbb{R}^{m} \mid y^{\top}w< 0 \text{ for all } w\in K^{*} \setminus \{0\}\}.$ \\

Let $C\subset K^{*} \setminus \{0\}$ be a compact set such that
\begin{equation}\label{con. for generator of K*}
    K^{*}=\mathrm{cone}(\mathrm{conv}(C)).
\end{equation}
%
%
%
%i.e., the polar cone $K^{*}$ is generated by the conic hull of $C$. 
Then, we say that $C$ is a generator of the cone $K^{*}$. \\

Note that if the set $K$ is a polyhedral cone, $K^{*}$ is also a polyhedral cone, and $C$ can be taken as the finite set of extremal rays of $K^{*}$. \\ 

If $K=\mathbb{R}^{m+}$, then $K^{*}=\mathbb{R}^{m+}=K$. Therefore, for $K=\mathbb{R}^{m+}$, in \eqref{con. for generator of K*}, we can take $C$ as the canonical basis of $\mathbb{R}^{m}$. \\

For the rest of the analysis presented in this paper, we consider a generator $C$ of $K^*$ as discussed in the following lemma. \emph{Importantly, note that throughout the paper, we do not restrict $K$ to be finitely generated.} \\

\begin{lemma}
    Let $e\in \mathrm{int}(K)$ be given. Then,
    \begin{equation}\label{gen_polar_cone}
    C:=\left\{w\in K^{*} ~\middle|~ w^{\top}e=1\right\}
\end{equation}
is compact and satisfies the relation \eqref{con. for generator of K*}.
\end{lemma}

\begin{proof}
Note that $e \in \mathrm{int}(K)$. Therefore, for any given $\bar x \in \mathbb{R}^m$ with $\|\bar x\| = 1$, there exists a $\gamma>0$ such that $e-\gamma \bar x \in K$. Further, for any $\tilde w \in C$, we have
\begin{align*}
   {\tilde w}^{\top}(e - \gamma \bar x) \geq 0 \implies \tfrac{1}{\gamma} \geq  {\tilde w}^{\top} \bar x  \implies \tfrac{1}{\gamma} \geq \underset{\|x\| = 1}{\max} {\tilde w}^{\top} x  \implies \tfrac{1}{\gamma} \geq \|\tilde w\|.  
\end{align*}
Thus, the set $C$ is bounded. Moreover, $w \mapsto w^{\top}e$ is a continuous mapping from $\mathbb{R}^{m}$ to $\mathbb{R}$, and the set $\{1\} $ is closed in $\mathbb{R}$. Therefore, the set $C$ is compact. \\

\noindent 
Next, note that $C \subseteq K^{*}$ and $K^{*}$ is a convex cone. So, 
$\mathrm{cone}(\mathrm{conv}(C)) \subseteq K^*$. \\

To prove $K^* \subseteq \mathrm{cone}(\mathrm{conv}(C))$, let $\bar{w}$ be a non-zero element of $K^{*}$. We show that there exists a $\bar{t}>0$ and $\bar{w}_{c}\in C$ such that $\bar{w} = \bar{t}~\bar{w}_{c}$. Then, it will imply that $\bar{w}$ belongs to $\mathrm{cone}(\mathrm{conv}(C))$ and the result will be followed. \\

\noindent
    Denote $\bar{t}:=\bar{w}^{\top}e$. Then, we note from the definition of $K^*$ that $\bar{t}>0$ since $\bar{w} \in K^{*}-\{0\} $ and $e \in \mathrm{int}(K)$. \\ 

\noindent
    Denoting $\bar{w}_{c}:= \tfrac{\bar{w}}{\bar{t}}$, we see that $\bar{w}_{c} \in K^{*}$ as $\bar{t}>0$, and $\bar{w}_{c} \in C$ because 
    $\tfrac{{\bar{w}}^{\top}e}{\bar{t}}=1$. \\ 

\noindent
    Thus, for any non-zero $\bar{w} \in K^{*}$, we have $\bar{t} = \bar{w}^\top e > 0$ and $\bar{w}_{c} = \tfrac{\bar{w}}{\bar{t}} \in C$ such that $\bar{w} = \bar{t}~\bar{w}_{c}$. Hence, the result follows. 
\end{proof}
%
%
%
%
%The following definition combines the notions of minimal and weakly minimal elements of a set associated with a generic cone.\\

Next, we define minimal elements of the sets in $\mathbb{R}^{m}$ with respect to the cone $K$.  \\

\begin{definition}\label{Def_minimal}
The set of minimal elements of a nonempty set $A \in \mathscr{P}(\mathbb{R}^{m})$ with respect to the cone $K$ is defined as
$$\mathrm{Min}(A,K) := \left\{ y \in A ~\middle|~ (y - K) \cap A = \{y\} \right\}.$$ 
Similarly, the set of weakly minimal elements of $A$ with respect to $K$ is defined as
$$\mathrm{WMin}(A,K) := \left\{ y \in A ~\middle|~ (y - \mathrm{int}(K)) \cap A = \emptyset \right\}.$$

\end{definition}

\begin{proposition}\label{min+A}
\emph{\cite{eichfelder2011vector}}
    Let $A \in \mathscr{P}(\mathbb{R}^{m})$ be compact. Then, $A$ satisfies the domination property with respect to $K$: 
\[ A + K = \mathrm{Min}(A, K) + K. \]
\end{proposition}

\smallskip 
\smallskip

Next, we discuss a preordering relation on \(\mathscr{P}(\mathbb{R}^{m})\), which is employed in this article to formulate the concept of optimality for the set optimization problems under consideration. For further insights into set ordering relations, we refer the reader to \cite{jahn2011new,karaman2018partial}. \\

\begin{definition}\emph{\cite{bouza2021steepest}}
    Let $A$ and $B$ belong to $ \mathscr{P}(\mathbb{R}^{m})$. For the given cone $K$, the lower set less relation $\preceq^{\ell}$ is defined on $\mathscr{P}(\mathbb{R}^{m})$ as follows:
    $$ A \preceq^{\ell} B \Longleftrightarrow B\subseteq A+K.$$
\noindent Similarly, for the solid cone $K$, the strict lower set less relation $\prec^{\ell}$ is defined by
 $$A \prec^{\ell} B \Longleftrightarrow B\subseteq A+\mathrm{int}(K).$$
\end{definition}

\smallskip
\smallskip

Now, we present the set optimization problem under consideration, accompanied by a solution concept based on the lower set less preordering relation $\preceq^{\ell}$. \\

Consider a set-valued mapping $F: \mathbb{R}^n \rightrightarrows \mathbb{R}^m$, where $F$ takes only nonempty values. A set optimization problem, under the preordering $\preceq^{\ell}$, with $F$ as the objective function is presented as follows: 
\begin{align}\label{SPL}
\tag{$\mathcal{SOP}_\ell$}
    \preceq^{\ell}\text{--- }\underset{x\in \mathbb{R}^{n}}{\min}F(x).
\end{align}

A solution concept for \eqref{SPL} is defined in the following manner. A point $\bar{x}\in \mathbb{R}^{n}$ is called a \emph{local weakly minimal solution} of \eqref{SPL} if there exists a neighbourhood $ U $ of $\bar{x}$ such that 
\begin{equation}\label{P_min_sol}
    \nexists~ x\in U: F(x)\prec^{\ell}F(\bar{x}).
\end{equation}
If \eqref{P_min_sol} holds with $U=\mathbb{R}^{n}$, then $\bar{x}$ is called a weakly minimal solution of \eqref{SPL}. \\  

Throughout the rest of the paper, we undertake the following assumption on the set optimization problem \eqref{SPL}.\\

\begin{framed}

\begin{assumption}\label{Ass on SPL}
    Let $e \in \mathrm{int}(K)$ be a given element and $f^1, f^2, \ldots, f^p : \mathbb{R}^n \rightarrow \mathbb{R}^m$ be continuously differentiable functions. We assume that the set-valued mapping $F$ in \eqref{SPL} is defined by
    \begin{equation*}
        F(x):=\left\{f^1(x), f^2(x), \ldots, f^p(x)\right\}.
    \end{equation*}
\end{assumption}    
\end{framed}

Below, we note down a few inequalities to simplify some calculations in the later part of the paper. \\

\begin{lemma}\label{S_calcul.}
\emph{\cite{lucambio2018nonlinear}}
For any real numbers $z_{1},z_{2},$ and $t \neq 0$, we have
    \begin{enumerate}[label=(\roman*)]
        \item $z_{1}z_{2} \leq \tfrac{z^2_{1}}{2}+\tfrac{z^2_{2}}{2}$,

        \item $2z_{1}z_{2} \leq 2 t ^2 z^{2}_{1}+\tfrac{z^2_{2}}{2t^2}$,

        \item $(z_{1}+z_{2})^2 \leq 2z^2_{1} +2 z^2_{2}$, \text{and }

        \item $(z_{1}+z_{2})^2 \leq (1+2t^2)z^2_{1} + \left(1+\tfrac{1}{2t^2}\right) z^2_{2}$.
    \end{enumerate}
\end{lemma}

\subsection{Optimality conditions}
In this subsection, we discuss a necessary condition for local weakly minimal solutions of the problem \eqref{SPL} under Assumption \ref{Ass on SPL}, which is identified from \cite{bouza2021steepest} with the help of the Gerstewitz scalarizing function. The mentioned necessary optimality condition (later in Remark \ref{ex_a_u} (ii)) is used as a stopping condition of the proposed algorithm in this paper.  To recall the necessary optimality condition, we require a few index-related set-valued mappings, which are defined after a short description of the Gerstewitz scalarizing function.  \\

\begin{definition}
    For a given element $e\in \mathrm{int}(K)$, the Gerstewitz function $\psi_{e} : \mathbb{R}^{m}\rightarrow \mathbb{R}$ associated with the element $e$ and the cone $K$ is defined by 
    \begin{equation}\label{Grit}
        \psi_{e}(y):=\min\left\{t\in \mathbb{R} ~\middle|~  te \in y+K\right\}.
    \end{equation}
\end{definition}

\begin{lemma}\label{pro_Ger.}
\emph{\cite{khan2016set}} 
Let $y\text{ and } z$ belong to $\mathbb{R}^{m}$,  and $\alpha $ be a nonnegative real number. Then,
     \begin{enumerate}[label=(\roman*)]
     \item $\psi_{e}(y+z) \leq \psi_{e}(y)+\psi_{e}(z)$ and $\psi_{e}(\alpha y) = \alpha ~ \psi_{e}(y)$.

     \item If $y\preceq_{K}z$, then $\psi_{e}(y)\leq \psi_{e}(z)$, and if $y\prec_{K}z$, then $\psi_{e}(y)< \psi_{e}(z)$.

     \item $\psi_{e}$ is Lipschitz continuous on $\mathbb{R}^{m}$.

     \item $\psi_{e}$ satisfies the following representability properties:
     $$-K=\left\{y\in \mathbb{R}^{m} ~\middle|~ \psi_{e}(y)\leq 0\right\}\text{ and }-\mathrm{int}(K)=\left\{y\in \mathbb{R}^{m} ~\middle|~ \psi_{e}(y)< 0\right\}.$$
\end{enumerate}    
\end{lemma}

We now define a few index-related set-valued mapping. \\

\begin{definition}\emph{\cite{bouza2021steepest}} 
\begin{enumerate}[label=(\roman*)]
\item The active-index function of minimal elements $ I : \mathbb{R}^n \rightrightarrows [p]$ associated with $F$ is     
   \[ I(x) := \left\{ i \in [p] ~\middle|~ f^i(x) \in \mathrm{Min}(F(x), K)\right \}.\]
  In the same way, for weakly minimal elements, the active-index function $ I_{W}: \mathbb{R}^n \rightrightarrows [p]$ associated with $F$ is defined by
\begin{equation*}
     I_{W}(x) := \left\{ i \in [p] ~\middle|~ f^i(x) \in \mathrm{WMin}(F(x), K)\right \}.
\end{equation*}

\item For a vector $v \in \mathbb{R}^m$, we define $I_v : \mathbb{R}^n \rightrightarrows [p]$  as
\begin{equation*}
   I_v(x) := \left\{ i \in I(x) ~\middle|~ f^i(x) = v \right\}. 
\end{equation*}
 
\end{enumerate}
\end{definition}

Note that $I_v(x) = \emptyset$ whenever $v \notin \text{Min}(F(x), K)$, and for any $x\in \mathbb{R}^{n}$, \\
\begin{enumerate}[label=(\roman*)]
     \item  
     $I(x)=\underset{v \in \text{Min}(F(x),K)}{\bigcup} I_{v}(x)$, and 
     \item for $u\neq v$, $I_u(x)\cap I_v(x)=\emptyset.$ \\
\end{enumerate}

\begin{definition}\emph{\cite{bouza2021steepest}}
    The cardinality function $\omega: \mathbb{R}^n \rightarrow \mathbb{R}$  of the set of minimal elements of $F$ is $$\omega(x) := |\mathrm{Min}(F(x), K)|. $$ Additionally, we denote $\bar{\omega} = \omega(\bar{x})$. \\ 
\end{definition}

%Henceforth, we assume that for any point $x \in \mathbb{R}^n$, an enumeration $\{ v^{x}_{1}, v^{x}_{2}, \ldots, v^{x}_{\omega(x)} \}$ of the set $\text{Min}(F(x), K)$ has been predetermined.\\

\smallskip

\begin{definition}\emph{\cite{bouza2021steepest}}
    Let $x\in\mathbb{R}^n$, and $\{v_1^x,v_2^x,\ldots,v^x_{\omega(x)}\}$ be an enumeration of the set $\mathrm{Min}(F(x),K)$. The partition set at $x$ is defined as
   \[ P_x:=\prod_{j=1}^{\omega(x)} I_{v_j^x}(x). \]
\end{definition}

Utilizing the concept of partition set, we now recall a result that helps us to verify whether a given point \(\bar{x} \in \mathbb{R}^{n}\) is a weakly minimal solution of the considered set optimization problem with the help of a class of \emph{special} multiobjective optimization problems. \\

\begin{lemma} \emph{\cite{bouza2021steepest}}\label{Vectorization} 
    Let $\bar x$ be a given point in $\mathbb{R}^n$. Suppose that 
    \begin{enumerate}[label=(\roman*)]
        \item $\tilde {K}\in \mathscr{P}(\mathbb{R}^{m\bar{\omega}})$ is the cone  
        $ \tilde {K} : = K \times K \times \cdots \times K$ ($\bar{\omega}$ times).

        \item $P_{\bar{x}}$ is the partition set at $\bar{x}$.

        \item For any $a:=(a_{1},a_{2}, \ldots, a_{\bar{\omega}})\in P_{\bar{x}}$, the function $\tilde{f}^{a}:\mathbb{R}^{n}\rightarrow \prod_{j=1}^{\bar{\omega}}\mathbb{R}^{m}$ is defined as 
\begin{align*}
\tilde{f}^{a}(x):=\left(f^{a_1}(x), f^{a_2}(x), \ldots, f^{a_{\bar{\omega}}}(x)\right)^\top.
\end{align*}
\end{enumerate}
Then, $\bar{x}$ is a local weakly minimal solution of \eqref{SPL} if and only if for every $a\in   P_{\bar{x}}$, $\bar{x}$ is a local weakly minimal solution of the multiobjective optimization problem
    \begin{align}\label{MOP_x_k}
    \tag{$\mathcal{MOP}_{\tilde{K}}$}
    \preceq^{\tilde{K}}\text{---}\underset{x\in \mathbb{R}^{n}}{\min}\tilde{f}^{a}(x).
\end{align}
\end{lemma}

\smallskip

Lemma \ref{Vectorization} allows us to determine whether a point \(\bar{x} \in \mathbb{R}^{n}\) is a local weakly minimal point of the problem \eqref{SPL} by verifying if \(\bar{x}\) is a local weakly minimal point of all the multiobjective optimization problems \eqref{MOP_x_k} corresponding to the partition set $P_{\bar x}$. Thus, a question may naturally arise: Can we be able to identify a local weakly minimal point of the problem \eqref{SPL} by solving the multiobjective optimization problems \eqref{MOP_x_k} for all $a$ in $P_{\bar x}$? The answer is evidently `yes' provided $\bar x$ is given. However, the practical question is `how to identify $\bar x$ through a systematic way so that $\bar x$ is a local weakly minimal point of \eqref{MOP_x_k} for all $a$ in $P_{\bar x}$?' To answer this question, note that 
\begin{enumerate} 
\item[(i)] to employ Lemma \ref{Vectorization} for identifying a local weakly minimal point of \eqref{SPL}, we need to first guess a point $\bar x$, which is not a trivial task, and 
\item[(ii)] the construction of the objective function of \eqref{MOP_x_k} exclusively depends on the partition set $P_{\bar{x}}$ at $\bar x$.  
\end{enumerate} 
To identify or guess a point $\bar x$ which may satisfy Lemma \ref{Vectorization}, we commonly aim to generate a sequence of iterates $\{x_k\}$ whose one of the limit points is $\bar x$. In generating the sequence $\{x_k\}$, if the current iterate ${x}_k$ is not a local weakly minimal solution of \eqref{MOP_x_k} corresponding to at least one element of $P_{x_k}$, then by Lemma \ref{Vectorization}, $x_k$ is not a local weakly minimal point of \eqref{SPL}. So, we must proceed to finding the next iterate $x_{k+1}$. As $x_{k + 1}$ is different from ${x}_k$, the partition set $P_{x_{k+1}}$ at $x_{k + 1}$ may be different from the partition set $P_{x_k}$. Accordingly, the collection of problems \eqref{MOP_x_k} at $x_k$ may be different than that at $x_{k+1}$. Therefore, in general, identification of a local weakly minimal point of \eqref{SPL} does not amount to solving just a collection of multiobjective optimization problems \eqref{MOP_x_k} corresponding to all the elements of the partition set $P_{\bar x}$. In this paper, in fact, we will see that we need to solve a sequence of a class of multiobjective optimization problems to identify a weakly minimal point of \eqref{SPL}. \\

Based on Lemma \ref{Vectorization}, we discuss below a necessary condition for weakly minimal points of \eqref{SPL} using the concept of stationary points. \\

\begin{definition}\emph{\cite{kumar2024nonlinear}}\label{Stat. cond.} 
A point $\bar{x}$ is said to be a stationary point of \eqref{SPL} if 
    $$\forall~ a\in P_{\bar{x}} \text{ and } d\in \mathbb{R}^{n},~ \exists~ j\in [\bar{\omega}]:~\nabla f^{a_{j}}(\bar{x})^{\top}d\notin -\mathrm{int}(K).$$ 
\end{definition}

Notice from Theorem 3.1 and Proposition 3.1 of \cite{bouza2021steepest} that a weakly minimal point of \eqref{SPL} is necessarily a stationary point of \eqref{SPL}. \\ 

It is important to note that in this work, the primary goal is to find weakly minimal solutions of the problem \eqref{SPL}. However, the identification of weakly minimal points poses significant computational challenges due to the difficulty of finding a computationally viable stopping condition that ensures the weak minimality of a point. Therefore, we end up finding stationary points of \eqref{SPL}, similar to the conventional first-order optimization methods. \\

For the identification of a computationally viable necessary condition for weakly minimal points of \eqref{SPL}, we consider a parametric family of functions $\varphi_{x}: ~P_x\times \mathbb{R}^{n}\rightarrow \mathbb{R}$, $x \in \mathbb{R}^n$, with the help of Gerstewitz scalarizing function \eqref{Grit}, as follows:
\begin{align}\label{varphhi}
 \varphi_{x}(a,d) :=\mathcal{F}^{a}(x,d)+\tfrac{1}{2}\|d\|^{2},
\end{align}
 where  $a\in P_x,~d\in \mathbb{R}^{n}$, and the expression of $\mathcal{F}^{a}(x,d)$ is given by
\begin{align}\label{F_abar}
    \mathcal{F}^{a}(x,d):=\underset{j\in [\omega(x)]}{\max} ~ \psi_{e}\left(\nabla f^{a_j}(x)^{\top}d\right).
\end{align}
Note that for each $x\in \mathbb{R}^n$ and $a\in P_{x}$, the function $\varphi_{x}(a,\cdot)$ is strongly convex in $\mathbb{R}^n$. Therefore, the function $\varphi_{x}(a,\cdot)$ attains a unique minimum over $\mathbb{R}^n$. Also, for any $a \in P_{x}$,
\begin{equation}\label{dg_aux}
    \underset{d\in \mathbb{R}^n}{\min}\varphi_{x}(a,d) \leq \varphi_{x}(a,0) = 0.
\end{equation}
Moreover, if $d_{x} \in \mathbb{R}^{n}$ is such that $\varphi_{x}(a,d_{x})= \underset{d\in \mathbb{R}^n}{\min}\varphi_{x}(a,d)$, then
\begin{equation*}
     \underset{d\in \mathbb{R}^n}{\min}\varphi_{x}(a,d)=0 \Longleftrightarrow d_{x} = 0.
\end{equation*}
 Since $P_{x}$ is a finite set, the minimum of $\varphi_{x}$ exists over the set $P_x\times \mathbb{R}^{n}$. Consequently, we define a function $\phi:\mathbb{R}^{n}\rightarrow \mathbb{R}$ by 
\begin{align}\label{phii}
  \phi(x) :=\underset{(a,d)\in P_{x}\times\mathbb{R}^{n}}{\min} ~ \varphi_{x}(a,d).  
\end{align}
Notice from \eqref{dg_aux} that for all $x \in \mathbb{R}^n$, $\phi(x) \leq 0$. Also, if $(a_{x},d_{x})\in P_{x}\times \mathbb{R}^n$ is such that $\phi(x)=\varphi_{x}(a_{x},d_{x})$, then
\begin{equation}\label{Cond_req_Nec}
    \phi(x) =0 \Longleftrightarrow d_{x} = 0.
\end{equation}
\begin{framed}
In the rest of the paper, we denote  
\begin{align*}
     (a_{x},u_{x}) \in \argmin_{(a, d) \in P_x \times \mathbb{R}^n} ~ \varphi_{x}(a,d),~~\text{i.e.},~ 
     \phi(x) = \varphi_{x}(a_{x},u_{x}),
\end{align*}
and 
\begin{enumerate} 
\item[(i)] at a point $\bar{x}$, we use the notations $\bar{a}$, $\bar{u}$, and $\bar{\mathcal{F}}(\cdot,\cdot)$ instead  of $a_{\bar{x}}$, $u_{\bar{x}}$, and $\mathcal{F}^{\bar{a}}(\cdot,\cdot)$, respectively, and 
\item[(ii)] for an iterative sequence $\{x_k\}$, for any $k \in \mathbb{N}$, we use the notations $P_k$, $\omega_k$, $a_k$, $u_k$, and $\mathcal{F}^{k}(\cdot,\cdot)$ instead of $P_{x_k}$, $\omega(x_k)$, $a_{x_k}$, $u_{x_k}$, and $\mathcal{F}^{{a}_k} (\cdot,\cdot)$, respectively. 
\end{enumerate}
\end{framed}

To find a necessary condition of local weakly minimal points of \eqref{SPL}, let  $\bar{x}$ be a local weakly minimal point of \eqref{SPL}. Then, as $\bar{x}$ is a stationary point, we have from Definition \ref{Stat. cond.} that for any $a\in P_{\bar{x}}$ and $d \in \mathbb{R}^{n}$ there exists $j\in [\bar{\omega}]$ such that $\psi_{e}\left(\nabla f^{\bar{a}_{j}}(\bar{x})^{\top}d\right)\geq 0$. So, at $\bar{x}$, we have $\bar{\mathcal{F}}(\bar{x},d)\geq 0$ for all $d\in \mathbb{R}^{n}$. Therefore, from \eqref{varphhi}, we have 
\[\varphi_{\bar{x}}(\bar{a},d) =\bar{\mathcal{F}}(\bar{x},d)+\tfrac{1}{2}\|d\|^{2}\geq 0,\] and hence $\phi(\bar x) \ge 0$. Thus, by (\ref{F_abar}), at a local weakly minimal point $\bar x$ of \eqref{SPL}, we necessarily have $\phi(\bar x) = 0$. Consequently, from \eqref{Cond_req_Nec}, we get $\bar{u}=0$. Accumulating all, we have the following remark. \\

\begin{remark}\label{ex_a_u}
Consider the problem \eqref{SPL}.
\begin{enumerate}[label=(\roman*)]

    \item  Let $\bar{x} \in \mathbb{R}^n$. Then, there exists $(\bar{a},~\bar{u})\in P_{\bar{x}} \times \mathbb{R}^{n}$ such that $\phi(\bar{x})=\varphi_{\bar{x}}(\bar{a},\bar{u})$.

    \item \emph{(Necessary condition of local weakly minimal points).} Let $\bar{x} \in \mathbb{R}^n$ be a local weakly minimal point of \eqref{SPL} and $(\bar{a},\bar{u})\in P_{\bar{x}} \times \mathbb{R}^{n}$ be such that $\phi(\bar{x})=\varphi_{\bar{x}}(\bar{a},\bar{u})$.  Then, $\bar{u}=0.$ \\ 
    
\end{enumerate}
\end{remark}

\begin{proposition}\label{Pro. station. cond.}
\emph{\cite{bouza2021steepest}} 
Consider the functions $\varphi_{\bar{x}}$ and $\phi$ as given in \eqref{varphhi} and \eqref{phii}, respectively. Let $(\bar{a},\bar{u})\in P_{\bar{x}}\times \mathbb{R}^{n}$ be such that $\phi(\bar{x})=\varphi_{\bar{x}}(\bar{a},\bar{u})$. Then, the following results hold:
\begin{enumerate}[label=(\roman*)]
    \item If $\bar{x}$ is a stationary point of \eqref{SPL}, then $\bar{u}=0$ and $\phi(\bar{x})=0$.

    \item If $\bar{x}$ is a nonstationary point of \eqref{SPL}, then $\bar{u}\neq 0$ and $\phi(\bar{x})<0$.  \\ 
\end{enumerate}
\end{proposition}

From Remark \ref{ex_a_u} (ii) and Proposition \ref{Pro. station. cond.}, we note that we can devise a stopping condition of a numerical scheme $\{x_k\}$ to identify a local weakly minimal point of \eqref{SPL} as follows. At the current iterate $x_k$, we identify a point $(a_k, u_k) \in P_{k} \times \mathbb{R}^n$ so that 
\[(a_k, u_k) \in \argmin_{(a, d) \in P_{k} \times \mathbb{R}^n} ~ \varphi_{x_k}(a,d). \]
If $u_k = 0$, then we stop generating successive points since $x_k$ satisfies a necessary condition for local weak minimality. If, however, $u_k \neq 0$, then in the following section, we show that $u_k$ is a $K$-descent direction for $F$ at $x_k$. 
%Accordingly, we find the next iterate along the direction $u_k$ with the help of a line search. 
So, a computationally viable stopping condition of the numerical scheme $\{x_k\}$ can be $\|u_k\| < \varepsilon$ for a given precision scalar $\varepsilon > 0$. \\

For generating a sequence $\{x_k\}$ that aims to identify a weakly minimal point of \eqref{SPL}, in this study, we follow the common process as in conventional optimization problems. Starting with an initial point $x_0$, the iterates are generated by 
\begin{equation}\label{line.str}
    x_{k+1}:=x_{k}+\alpha_{k} d_{k}, ~ k=0,1,2,\ldots, 
\end{equation}
where $d_k$ is a $K$-descent direction of the set-valued mapping $F$ at $x_k$ and $\alpha_k$ is a suitable step-length along $d_k$, determined by a line search procedure. \\

In the next section, for the set-valued objective function $F$ of \eqref{SPL}, we discuss the notion of $K$-descent direction and introduce the concept of Wolfe line search techniques to generate a suitable step-size $\alpha_k$. Later, in Section \ref{section4}, we propose a conjugate gradient scheme for identifying $K$-desccent directions for set-valued functions.

\section{$K$-descent direction and Wolfe line searches }\label{sec_Ana_wofe}

% This section discusses descent directions and the Wolfe line search conditions for the set optimization problem \eqref{SPL}. It begins by introducing the concept of the descent direction using a specific element from the partition set at a given point. Next, the Wolfe line search conditions for \eqref{SPL} is proposed. It is then shown that at a nonstationary point of \eqref{SPL}, a step-size exists that satisfies the proposed Wolfe line search conditions along a $K$-descent direction. Finally, the Zoutendijk-like condition is proven, which helps in proving the global convergence of the proposed method. \\ 

We define the concept of a descent direction for the set-valued objective function $F$ of the problem \eqref{SPL} in the following way. Suppose $\bar{x}$ is a nonstationary point of \eqref{SPL}. Then, by Definition \ref{Stat. cond.}, there exists a $d\in \mathbb{R}^{n}$ such that 
\begin{align*}
\forall~ j\in [\bar{\omega}]: \nabla f^{\bar{a}_{j}}(\bar{x})^{\top}d \in -\mathrm{int}(K) ~~\Longleftrightarrow~~ \bar{\mathcal{F}}(\bar{x}, d)<0.
\end{align*}
Consequently, for sufficiently small $\alpha>0$, we have
\[\forall~ j\in [\bar{\omega}]: ~f^{\bar{a}_{j}}(\bar{x}+\alpha d) \prec_{K} f^{\bar{a}_{j}}(\bar{x}). \]
Thus, we have
\begin{align*}
    F(\bar{x}) &\subseteq \left\{f^{\bar{a}_{1}}(\bar{x}),f^{\bar{a}_{2}}(\bar{x}),\ldots, f^{\bar{a}_{\bar{\omega}}}(\bar{x})\right\}+K\\
    &\subseteq \left\{f^{\bar{a}_{1}}(\bar{x}+\alpha d),f^{\bar{a}_{2}}(\bar{x}+\alpha d),\ldots, f^{\bar{a}_{\bar{\omega}}}(\bar{x}+\alpha d)\right\} + \mathrm{int}(K)+ K\\
     &\subseteq \left\{f^{\bar{a}_{1}}(\bar{x}+\alpha d),f^{\bar{a}_{2}}(\bar{x}+\alpha d),\ldots, f^{\bar{a}_{\bar{\omega}}}(\bar{x}+\alpha d)\right\} + \mathrm{int}(K),
\end{align*}
which implies $
F(\bar{x}+\alpha d)  \prec^{\ell} F(\bar{x}).$ Accordingly, a descent direction of $F$ is defined as follows. \\

\begin{definition}\label{Definition of K_descent}
Consider the functions $\varphi_{\bar{x}}$ and $\phi$ as given in \eqref{varphhi} and \eqref{phii}, respectively. Let $(\bar{a},\bar{u})\in P_{\bar{x}}\times \mathbb{R}^{n}$ be such that $\phi(\bar{x})=\varphi_{\bar{x}}(\bar{a},\bar{u})$. Then, at a point $\bar{x}$, a direction $d \in \mathbb{R}^{n}$ 
\begin{enumerate} 
\item[(i)] is said to be a $K$-descent direction of $F$ if \begin{equation}\label{descent}
      \bar{\mathcal{F}}(\bar{x},d)< 0, \text{ and } 
   \end{equation} 
\item[(ii)] is said to be satisfying a sufficient descent condition if there exists $c>0$ such that 
   \begin{align}\label{suf_d}
   \bar{\mathcal{F}}(\bar{x}, d)\leq c~\bar{\mathcal{F}}(\bar{x},\bar{u}).   
   \end{align}
\end{enumerate}
\end{definition}

Notice that if $x_k$ is a nonstationary point of the objective function $F$ of \eqref{SPL}, then $\phi(x_k) < 0$. If $(a_k,u_k)\in P_{k} \times \mathbb{R}^{n}$ be such that $\phi({x}_k)=\varphi_{{x}_k}({a}_k,{u}_k)$, then $${\mathcal{F}^{k}}({x}_k, {u}_k) < - \tfrac{1}{2} \|\bar{u}_k\|^{2}<0.$$ Therefore, ${u}_k$ is a $K$-descent direction of $F$ at ${x}_k$. 
%So, from Remark \ref{ex_a_u}, we see that if $x_k$ is a weakly minimal point of \eqref{SPL}, then $u_k = 0$, and if $u_k \neq 0$, then $u_k$ is a $K$-descent direction of \eqref{SPL} at the current point $x_k$.  
\\

Next, we discuss line search techniques along a $K$-descent direction of $F$. For a continuously differentiable vector-valued function $f:\mathbb{R}^{n}\rightarrow \mathbb{R}^{m}$, along a $K$-descent direction $d$ at $\bar{x}$, Drummond et al.\ \cite{drummond2005steepest} extended the Armijo condition for the steepest descent method as follows: for $0<\rho<1$, there exists an $\bar{\alpha}>0$ such that
\begin{equation}\label{line search for steepest}
    \forall \alpha \in (0,\bar{\alpha}):  f(\bar{x}+\alpha d)\prec_{K} f(x) + \rho \alpha \nabla f(\bar{x})^{\top}d.
\end{equation}
In the context of the set optimization problem \eqref{SPL}, Bouza et al. \cite{bouza2021steepest} extended the Armijo condition \eqref{line search for steepest} for the set-valued map $F$ as follows: for $0<\rho<1$, there exists an $\bar{\alpha}>0$ such that 
\[ \forall \alpha \in (0,\bar{\alpha}): F(\bar{x}+\alpha \bar{u})\preceq^{{\ell}}\left\{f^{\bar{a}_{j}}(\bar{x})+\rho \alpha \nabla f^{\bar{a}_{j}}(\bar{x})^{\top} \bar{u}\right\}_{j\in [\bar{\omega}]}\prec^{{\ell}}F(\bar{x}).\]

\noindent In \cite{lucambio2018nonlinear}, the conventional  Wolfe line search conditions have been extended for vector optimization and utilized for analyzing the convergence of nonlinear conjugate gradient methods. Based upon the Wolfe line search in \cite{lucambio2018nonlinear}, we extend it for the set optimization problem \eqref{SPL} as follows. \\

\begin{definition}
    Let $\bar{d}$ be a $K$-descent direction for $F$ at $\bar{x}$. Consider $e \in \mathrm{int}(K)$ and $0<\rho<\sigma<1$. We say that $\alpha>0$ satisfies the standard Wolfe conditions if 
\begin{subequations}\label{Arm with stan. wol.}
    \begin{align}
       & F(\bar{x}+\alpha \bar{d})\preceq^{{\ell}}\left\{f^{\bar{a}_{j}}(\bar{x})+\rho \alpha \bar{\mathcal{F}}(\bar{x},\bar{d})e\right\}_{j\in [\bar{\omega}]}\prec^{{\ell}}F(\bar{x}) \label{Arm in Stand.} \\
       \noalign{\noindent \text{and}}
       & \bar{\mathcal{F}}(\bar{x}+\alpha\bar{d},\bar{d})\geq \sigma\bar{\mathcal{F}}(\bar{x},\bar{d}).\label{stan, wol.}
    \end{align}
\end{subequations}
We say that $\alpha>0$ satisfies the strong Wolfe conditions if 
\begin{subequations}\label{Arm with strong. wol.}
    \begin{align}
       & F(\bar{x}+\alpha \bar{d})\preceq^{{\ell}}\left\{f^{\bar{a}_{j}}(\bar{x})+\rho \alpha \bar{\mathcal{F}}(\bar{x},\bar{d})e\right\}_{j\in [\bar{\omega}]}\prec^{{\ell}}F(\bar{x})  \label{Armi_con}  \\
       \noalign{\noindent \text{and}}
       & \left\lvert \bar{\mathcal{F}}(\bar{x}+\alpha\bar{d},\bar{d}) \right\rvert \leq \sigma \left \lvert \bar{\mathcal{F}}(\bar{x},\bar{d})\right\rvert.\label{strong. wol.}
    \end{align}
\end{subequations}
\end{definition}

\begin{framed} 
In \cite{lucambio2018nonlinear}, the authors extended the traditional Wolfe line search conditions for vector optimization only for the case where \(K\) is finitely generated and demonstrated the existence of a step length that satisfies these Wolfe conditions.  However, \emph{it remains an open problem to extend the Wolfe line search for a nonfinitely generated cone $K$, as mentioned in} \cite{lucambio2018nonlinear}. It is important to note that in this entire study, we are working without assuming that $K$ is finitely generated. In the following theorem, we attempt to address this open problem using the Gerstewitz scalarizing function and a generator $C$ for the cone $K^{*}$ given in \eqref{Grit} and  \eqref{gen_polar_cone}, respectively.  
\end{framed}
\smallskip
\smallskip

\begin{theorem}\label{exi.step}
    Suppose that $\bar{x}$ is a nonstationary point of \eqref{SPL} and $(\bar{a},\bar{u})\in P_{\bar{x}}\times \mathbb{R}^{n}$ is an element such that $\phi(\bar{x})=\varphi_{\bar{x}}(\bar{a},\bar{u})$, where $\phi$ and $\varphi_{\bar{x}}$ are given in \eqref{phii} and \eqref{varphhi}, respectively. Furthermore, let $\bar{d}$ be a $K$-descent direction of $F$ at $\bar{x}$ and there exists an $\mathcal{A}\in \mathbb{R}^{m}$ such that for all $j\in [\bar{\omega}]$ and $\alpha>0$,
    \begin{equation}\label{bou_fun_nbd}
        \mathcal{A}\prec_{K} f^{\bar{a}_{j}}(\bar{x}+\alpha \bar{d}).
    \end{equation}
Then, for any given $\rho \in (0,1),~\sigma \in (\rho, 1)$, and $e\in \mathrm{int}(K)$, there exists an  interval $I\subset \mathbb{R}^{+}$ such that any $\alpha \in I$ satisfies the strong Wolfe conditions \eqref{Arm with strong. wol.}. 
    \end{theorem}

\begin{proof}
We divide the proof for the condition \eqref{Armi_con} into two steps. \\ 
    
\noindent
Step 1: \emph{We show that there exists an $\tilde{\alpha}>0$ such that for all $\alpha \in (0,\tilde{\alpha}]$},
\begin{equation}\label{Arm.vec}
f^{\bar{a}_{j}}(\bar{x}+\alpha \bar{d})\preceq_{K}f^{\bar{a}_{j}}(\bar{x})+\rho \alpha \bar{\mathcal{F}}(\bar{x},\bar{d})e \text{ for all } j\in [\bar{\omega}]. 
\end{equation}
    Assume on contrary to \eqref{Arm.vec} that there exists a sequence $\{\alpha_{k}\}\subset \mathbb{R}^{+}$ and $\hat{j} \in[\bar{ \omega} ]$ such that $\alpha_{k}\rightarrow 0$ and for all $k \ge 0$, 
    \begin{align*}
   f^{\bar{a}_{\hat{j}}}(\bar{x}+\alpha_{k} \bar{d})-f^{\bar{a}_{\hat{j}}}(\bar{x})-\rho \alpha_{k} \bar{\mathcal{F}}(\bar{x},\bar{d})e\notin -K.
   \end{align*}
   Note that $\mathbb{R}^{m}\setminus(-K) \text{ is a cone and }\alpha_{k}>0$. Therefore,
    \begin{align*}
   \tfrac{f^{\bar{a}_{\hat{j}}}(\bar{x}+\alpha_{k} \bar{d})-f^{\bar{a}_{\hat{j}}}(\bar{x})}{\alpha_{k}}-\rho \bar{\mathcal{F}}(\bar{x},\bar{d})e\notin -\mathrm{int}(K)  \text{ for all } k \ge 0.
    \end{align*}
  Thus, if $ k\rightarrow \infty$, then we get
   \begin{align}\label{Arm. contra}
   \nabla {f^{\bar{a}_{\hat{j}}}{(\bar{x})}}^{\top}\bar{d}-\rho \bar{\mathcal{F}}(\bar{x},\bar{d})e\notin -\mathrm{int}( K).
    \end{align}
    However, from Lemma \ref{pro_Ger.} (i) and the definition of $\bar{\mathcal{F}}$ in \eqref{F_abar}, we have
    \begin{align*}
       \psi_{e} \left({\nabla {f^{\bar{a}_{\hat{j}}}{(\bar{x})}}^{\top}\bar{d}-\rho \bar{\mathcal{F}}(\bar{x},\bar{d})e}\right)&\leq  \psi_{e}\left( {\nabla {f^{\bar{a}_{\hat{j}}}{(\bar{x})}}^{\top}\bar{d}}\right)-\rho \bar{\mathcal{F}}(\bar{x},\bar{d})\psi_{e}(e)\\
       &= \psi_{e}\left( {\nabla {f^{\bar{a}_{\hat{j}}}{(\bar{x})}}^{\top}\bar{d}}\right)-\rho \bar{\mathcal{F}}(\bar{x},\bar{d})\\
       &= \psi_{e}\left( {\nabla {f^{\bar{a}_{\hat{j}}}{(\bar{x})}}^{\top}\bar{d}}\right)-\rho~\underset{j\in[\bar{\omega}]}{\max}\left\{ \psi_{e} \left({\nabla {f^{\bar{a}_{j}}{(\bar{x})}}^{\top}\bar{d}}\right)\right\}\\
       &\leq  \psi_{e}\left( {\nabla {f^{\bar{a}_{\hat{j}}}{(\bar{x})}}^{\top}\bar{d}}\right)-\rho  \psi_{e}\left( {\nabla {f^{\bar{a}_{\hat{j}}}{(\bar{x})}}^{\top}\bar{d}}\right)\\
       &=(1-\rho)\psi_{e}\left( {\nabla {f^{\bar{a}_{\hat{j}}}{(\bar{x})}}^{\top}\bar{d}}\right)< 0
    \end{align*}
    because $0<\rho<1$ and $\psi_{e}\left( {\nabla {f^{\bar{a}_{\hat{j}}}{(\bar{x})}}^{\top}\bar{d}}\right)\leq \bar{\mathcal{F}}(\bar{x},\bar{d})<0.$ Hence,
\begin{align*}  
\nabla {f^{\bar{a}_{\hat{j}}}{(\bar{x})}}^{\top}\bar{d}-\rho \bar{\mathcal{F}}(\bar{x},\bar{d})e\in -\mathrm{int}(K),\end{align*} 
which is contradictory to \eqref{Arm. contra}. Thus, (\ref{Arm.vec}) is followed. \\ 
     
\noindent Step 2: \emph{We show that there exists an interval $I_{0}\subset[0,\infty)$ such that for all $\alpha \in I_0$, } 
    \begin{align*}
        F(\bar{x}+\alpha \bar{d})\preceq^{{\ell}}\left\{f^{\bar{a}_{j}}(\bar{x})+\rho \alpha \bar{\mathcal{F}}(\bar{x},\bar{d})e\right\}_{j\in [\bar{\omega}]}\prec^{{\ell}}F(\bar{x}).
    \end{align*}

    \noindent Observe that the left-hand side of \eqref{Arm.vec} is bounded below because of \eqref{bou_fun_nbd} and the right-hand side of \eqref{Arm.vec} is unbounded below in $[0,\infty)$. Therefore, the relation \eqref{Arm.vec} does not not hold for all $\alpha\in [0,\infty)$. Moreover, for all $j\in [\bar{\omega}]$, the function 
    \[\alpha \longmapsto ~f^{\bar{a}_{j}}(\bar{x})+\rho \alpha \bar{\mathcal{F}}(\bar{x},\bar{d})e-f^{\bar{a}_{j}}(\bar{x}+\alpha \bar{d}) \]
    is continuous on $\mathbb{R}$, and from \eqref{Arm.vec}, all the values of this function belong to $K$ for all $\alpha \in [0,\tilde{\alpha}]$. Since $K$ is a closed cone and $[\bar{\omega}]$ is finite, there exists a maximal element $\bar{\alpha}\in [\tilde{\alpha},\infty)$ such that the relation \eqref{Arm.vec} holds for all $\alpha \in I_{0}= [0,\bar{\alpha}]$. That is, \eqref{Arm.vec} is not true for all $\alpha \in [0,\alpha^{  \diamond}]$, where $\alpha^{  \diamond} >\bar{\alpha}$ and for all $\alpha\in I_{0}$.\\ \\ 
    Using Proposition \ref{min+A} and \eqref{Arm.vec}, we obtain for any $\alpha\in I_{0}$ that
    \begin{align*}
    F(\bar{x})&\subseteq \left\{f^{\bar{a}_{1}}(\bar{x}),f^{\bar{a}_{2}}(\bar{x}),\ldots,f^{\bar{a}_{\bar{\omega}}}(\bar{x})\right\}+K\\
        &\subseteq \left\{f^{\bar{a}_{j}}(\bar{x})+\rho \alpha \bar{\mathcal{F}}(\bar{x},\bar{d})e\right\}_{j\in [\bar{\omega}]}+ \mathrm{int}(K) \text{ as }\rho>0 ,\alpha>0,\text{ and }\bar{\mathcal{F}}(\bar{x},\bar{d}) <0\\
        \smallskip
         &\subseteq \left\{f^{\bar{a}_{j}}(\bar{x}+\alpha \bar{d})\right\}_{j\in [\bar{\omega}]}+\mathrm{int}(K)\\
         \smallskip
         & \subseteq F(\bar{x}+\alpha \bar{d})+\mathrm{int}(K), 
    \end{align*}
    which concludes the proof for \eqref{Armi_con}. \\

 Next, we consider to prove \eqref{strong. wol.}.
     As $\bar{\alpha}$ is the largest element which holds the relation \eqref{Arm.vec} and $[\bar{\omega}]$ is finite, there exist an $\alpha_{k}\in [\bar{\alpha},\bar{\alpha}+\tfrac{1}{k}], \text{ for each } k\in \mathbb{N},$ and a $\bar{j}\in [\bar{\omega}]$ such that
\begin{align}\label{arm_com_N}
         f^{\bar{a}_{\bar{j}}}(\bar{x}+\alpha_{k} \bar{d})\succ_{K}f^{\bar{a}_{\bar{j}}}(\bar{x})+\rho \alpha_{k} \bar{\mathcal{F}}(\bar{x},\bar{d})e \text{ for all } k\in \mathbb{N}.
    \end{align}
   Since $\bar{j}\in [\bar{\omega}]$, from \eqref{Arm.vec}, we deduce that for all $w\in C$, 
    \begin{align}\label{arm_polar}
        w^{\top}\left\{f^{\bar{a}_{\bar{j}}}(\bar{x}+\alpha \bar{d})-f^{\bar{a}_{\bar{j}}}(\bar{x})-\alpha \rho \bar{\mathcal{F}}(\bar{x},\bar{d})e\right\}\leq 0 \text{ for all } \alpha\in [0,\bar{\alpha}].
    \end{align}
    Note that $C$ is compact. Therefore, from \eqref{arm_com_N} and \eqref{arm_polar}, there exists a $\bar{w}\in C$ such that
     \begin{align*}
         \bar{w}^{\top}\left\{f^{\bar{a}_{\bar{j}}}(\bar{x}+\bar{\alpha} \bar{d})-f^{\bar{a}_{\bar{j}}}(\bar{x})-\bar{\alpha} \rho \bar{\mathcal{F}}(\bar{x},\bar{d})e\right\}= 0.
     \end{align*}
     Since $w^{\top}e=1$ for all $w\in C$, we get 
     \begin{align}\label{con_minval}
         \bar{w}^{\top}f^{\bar{a}_{\bar{j}}}(\bar{x}+\bar{\alpha} \bar{d})-\bar{w}^{\top}f^{\bar{a}_{\bar{j}}}(\bar{x})-\bar{\alpha} \rho \bar{\mathcal{F}}(\bar{x},\bar{d})= 0.
     \end{align}
      Now, associated with $\bar{w}$, define two functions ${\mathcal{H}}_{\bar{w}}$ and $\mathcal{I}:\mathbb{R}\rightarrow\mathbb{R}$
      by
      \begin{align*}
         {\mathcal{H}}_{\bar{w}}(\alpha)=\bar{w}^{\top}f^{\bar{a}_{\bar{j}}}(\bar{x}+\alpha \bar{d}) \text{ and }  
         \mathcal{I}(\alpha)={\mathcal{H}}_{\bar{w}}(\alpha)-{\mathcal{H}}_{\bar{w}}(0)-\alpha \rho\bar{\mathcal{F}}(\bar{x},\bar{d}), \text{ respectively}. 
      \end{align*}
      Note that $\mathcal{I}$ is a continuously differentiable function and $\mathcal{I}(0)=\mathcal{I}(\bar{\alpha})$. Therefore, by Mean Value Theorem, there exists an $\hat{\alpha}\in (0,\bar{\alpha})$ such that 
   \begin{align*}
       {\mathcal{I}^{\prime}}(\hat{\alpha})=0, \text{ i.e., } \bar{w}^{\top}\left(\nabla f^{\bar{a}_{\bar{j}}}(\bar{x}+\hat{\alpha}\bar{d})^{\top}\bar{d}\right)=\rho\bar{\mathcal{F}}(\bar{x},\bar{d}).
   \end{align*}
   This implies,
   \begin{align*}
      ~&~ \underset{w\in C}{\max}~ w^{\top}\left(\nabla f^{\bar{a}_{\bar{j}}}(\bar{x}+\hat{\alpha}\bar{d})^{\top}\bar{d}\right)\geq\rho\bar{\mathcal{F}}(\bar{x},\bar{d})\\
      \text{or,} ~&~  \underset{w\in C}{\max}~ w^{\top}\left(-\nabla f^{\bar{a}_{\bar{j}}}(\bar{x}+\hat{\alpha}\bar{d})^{\top}\bar{d}+\rho\bar{\mathcal{F}}(\bar{x},\bar{d})e\right)\leq 0.
   \end{align*}
   As $-K=\{y\in \mathbb{R}^{m} \mid y^{\top}w\leq 0 \text{ for all } w\in K^{*}\}$, we have 
   \begin{align*}
      -\nabla f^{\bar{a}_{\bar{j}}}(\bar{x}+\hat{\alpha}\bar{d})^{\top}\bar{d}+\rho\bar{\mathcal{F}}(\bar{x},\bar{d})e\in -K. 
   \end{align*}
     Then, from Lemma \ref{pro_Ger.} (iv) and (i), we obtain
     \begin{align*}
        ~&~\psi_{e}\left(-\nabla f^{\bar{a}_{\bar{j}}}\left(\bar{x}+\hat{\alpha}\bar{d}\right)^{\top}\bar{d}+\rho\bar{\mathcal{F}}\left(\bar{x},\bar{d}\right)e\right)\leq 0
        \\
        \implies ~&~ \psi_{e}\left(\rho\bar{\mathcal{F}}\left(\bar{x},\bar{d}\right)e\right)-\psi_{e}\left(\nabla f^{\bar{a}_{\bar{j}}}\left(\bar{x}+\hat{\alpha}\bar{d}\right)^{\top}\bar{d}\right) \\ 
        ~&~
       \leq\psi_{e}\left(-\nabla f^{\bar{a}_{\bar{j}}}\left(\bar{x}+\hat{\alpha}\bar{d}\right)^{\top}\bar{d}+\rho\bar{\mathcal{F}}\left(\bar{x},\bar{d}\right)e\right)
       \leq 0\\
       \implies ~&~ \psi_{e}\left(\rho\bar{\mathcal{F}}\left(\bar{x},\bar{d}\right)e\right)\leq \psi_{e}\left(\nabla f^{\bar{a}_{\bar{j}}}\left(\bar{x}+\hat{\alpha}\bar{d}\right)^{\top}\bar{d}\right)\\
       \implies ~&~ \underset{j \in [\bar{\omega}]}{\max}\left\{ \psi_{e}\left(\nabla f^{\bar{a}_{\bar{j}}}\left(\bar{x}+\hat{\alpha}\bar{d}\right)^{\top}\bar{d}\right)\right\} \geq \psi_{e}\left(\rho\bar{\mathcal{F}}\left(\bar{x},\bar{d}\right)e\right)= \rho\bar{\mathcal{F}}\left(\bar{x},\bar{d}\right)\\
       \implies ~&~\bar{\mathcal{F}}\left(\bar{x}+\hat{\alpha}\bar{d},\bar{d}\right)\geq \rho\bar{\mathcal{F}}\left(\bar{x},\bar{d}\right). 
       \end{align*}
     From Lemma \ref{pro_Ger.} (iii), $\psi_e$ is a continuous function. So, $\bar{\mathcal{F}}(\bar{x}+\alpha\bar{d},\bar{d})$ is also a continuous function of $\alpha$. Hence, by the intermediate value property, there exists $\alpha^{*}\in (0,\bar{\alpha})$ such that 
     $$\bar{\mathcal{F}}\left(\bar{x}+{\alpha}^{*}\bar{d},\bar{d}\right)= \rho\bar{\mathcal{F}}\left(\bar{x},\bar{d}\right).$$
     Since $\bar{\mathcal{F}}\left(\bar{x},\bar{d}\right)<0$ and $\sigma\in (\rho,1)$, we obtain
     \begin{align}\label{eq_s&st}
       \sigma \bar{\mathcal{F}}\left(\bar{x},\bar{d}\right)<\bar{\mathcal{F}}\left(\bar{x}+{\alpha}^{*}\bar{d},\bar{d}\right)<0.  
     \end{align}
 Since $\bar{\mathcal{F}}(\bar{x}+\alpha\bar{d},\bar{d})$ is a continuous function, there exists a neighbourhood $I\subset I_{0}$ of $\alpha^{*}$ such that
 \begin{align}\label{use_Lk}
     \sigma \bar{\mathcal{F}}\left(\bar{x},\bar{d}\right)<\bar{\mathcal{F}}\left(\bar{x}+{\alpha}\bar{d},\bar{d}\right)<0  \text{ for all } \alpha\in I,
 \end{align}
 which concludes the proof.  
\end{proof}

\begin{remark}\label{existence_of_std_wolfe}
    Under the hypothesis of \emph{Theorem \ref{exi.step}}, we get a guarantee for the existence of an interval $I \subset \mathbb{R}^{+}$ in which the standard Wolfe conditions \eqref{Arm with stan. wol.} hold because for a given $K$-descent direction $\bar d$ of $F$ at $\bar x$, we have $\bar{\mathcal{F}}(\bar x, \bar d)<0$. Thus, \eqref{strong. wol.} implies that for all $\alpha \in I$, 
    \begin{align*} 
   ~ & ~ \left\lvert \bar{\mathcal{F}}(\bar{x}+\alpha\bar{d},\bar{d}) \right\rvert \leq - \sigma \bar{\mathcal{F}}(\bar{x},\bar{d}) \\ 
    \implies  & ~ \sigma \bar{\mathcal{F}}(\bar{x},\bar{d}) \le \bar{\mathcal{F}}(\bar{x}+\alpha\bar{d},\bar{d}), 
    \end{align*}
    which is the condition \eqref{stan, wol.}. \\ 
\end{remark}

For the rest of the analysis on the convergence of any sequence generated by Algorithm \ref{algo}, we take the following often-used assumptions.\\

\begin{framed}

\begin{assumption}\label{Ass. 3}
 The level set $\mathcal{L}:=\left\{x \in \mathbb{R}^{n}: F(x)\preceq^{\ell}F(x_{0})\right\}$ is bounded.   \\
\end{assumption}

    \begin{assumption}\label{lip_ass}
    The functions $\nabla f^{i}$, $i = 1, 2, \ldots, p$, are Lipschitz continuous on a nonempty open set $\Gamma$ containing $\mathcal{L}$ with a common Lipschitz constant $L>0$, i.e.,
    \[\forall~ i \in [p]  \text{ and } y,z \in \Gamma: ~\|\nabla f^{i}(y)-\nabla f^{i}(z)\|\leq L\|y-z\|.\]
\end{assumption}

\begin{assumption}\label{bddbe_ass}
    For any sequence $\{S_{k}\} \subset F(\mathcal{L})$ with $$S_{k+1}\preceq^{\ell}S_{k} \text{ for all } k = 0, 1, 2, \ldots,$$ there exists a bounded set $S\subseteq \mathbb{R}^{m}$ such that $S\preceq^{\ell}S_{k}$ for all $k = 0, 1, 2, \ldots$.\\
\end{assumption}

\end{framed}

\smallskip
\smallskip

It is worth highlighting that these three assumptions are natural extensions of those applied in conventional or vector optimization problems. Under Assumptions \ref{lip_ass} and \ref{bddbe_ass}, we show below that the general method \eqref{line.str} fulfills a condition akin to the conventional Zoutendijk’s criterion if the standard Wolfe conditions \eqref{Arm with stan. wol.} are employed for choosing the step-length $\alpha_k$ at the iterate $x_k$ for all $k = 0, 1, 2, \ldots$. This result on Zoutendijk-like condition is used later to show the global convergence of the proposed conjugate gradient methods. \\

\begin{theorem}\label{the. zoun.}
Suppose that \emph{Assumptions \ref{lip_ass} and \ref{bddbe_ass}} hold. Furthermore, assume that for the iterative scheme \eqref{line.str}, $d_{k}$ is a $K$-descent direction of $F$ at $x_{k}$, and $\alpha_{k}$ satisfies the standard Wolfe conditions \eqref{Arm with stan. wol.} for all $k = 0, 1, 2, \ldots$. Then,
    \begin{align}\label{Zounten.}
        \sum_{k\geq 0}{\frac{{\mathcal{F}^{k}}^{2}(x_{k},d_{k})}{\|d_{k}\|^{2}}}< +\infty .
    \end{align}
\end{theorem}

\begin{proof}
    From the standard Wolfe condition \eqref{stan, wol.} and the Lipschitz continuity of $\psi_{e}$, we have
    \begin{align*}
        (\sigma-1)\mathcal{F}^{k}(x_{k},d_{k})~&\leq~ \mathcal{F}^{k}(x_{k+1},d_{k})-\mathcal{F}^{k}(x_{k},d_{k})\\
        ~& =~\underset{j\in [\omega_{k}]}{\max}\left\{\psi_{e}\left(\nabla f^{a_{k,j}}(x_{k+1})^{\top}d_{k}\right)\right\}-\underset{j\in [\omega_{k}]}{\max}\left\{\psi_{e}\left(\nabla f^{a_{k,j}}(x_{k})^{\top}d_{k}\right)\right\}\\
        ~&\leq~ \psi_{e}\left(\nabla f^{a_{k,j_{k}}}(x_{k+1})^{\top}d_{k}\right)-\psi_{e}\left(\nabla f^{a_{k,j_{k}}}(x_{k})^{\top}d_{k}\right),\\
       &~~~~~ \text{ where } j_{k} \in \underset{j\in [\omega_{k}]}{\argmax}\left\{\nabla f^{a_{k,j}}(x_{k+1})^{\top}d_{k}\right\}\\
        ~&\leq~ \bar{L}\|\nabla f^{a_{k,j_{k}}}(x_{k+1})-\nabla f^{a_{k,j_{k}}}(x_{k})\|\|d_{k}\|,\\
        &~~~~~ \text{ where } \bar{L} \text{ is a Lipschitz constant of }\psi_{e}\\
        ~&\leq~ \bar{L}L\alpha_{k}\|d_{k}\|^{2} \text{ by Assumption \ref{lip_ass}}.
    \end{align*}
Thus, we obtain
\begin{align}\label{ineq.for.zau}
        \tfrac{{\mathcal{F}^{k}}^{2}(x_{k},d_{k})}{\|d_{k}\|^{2}}\leq \hat{L}\alpha_{k}\tfrac{\mathcal{F}^{k}(x_{k},d_{k})}{\sigma-1}, \text{ where }\hat{L}=\bar{L}L.
    \end{align}
    Consider a function $\mathscr{G}:\mathscr{P}(\mathbb{R}^{m})\rightarrow \mathbb{R}\cup\{-\infty\}$ defined by 
    $$\mathscr{G}(A)=\underset{z\in A}{\inf}~\psi_{e}(z), \text{ where } A\in \mathscr{P}(\mathbb{R}^{m}). $$
    Notice that $\psi_{e}$ has a monotonicity property. Therefore, corresponding to the preorder $\preceq^{\ell}$, the function $\mathscr{G}$ is monotone, i.e.,
    $$A_{1}\preceq^{\ell} A_{2} \implies \mathscr{G}(A_{1})\leq \mathscr{G}(A_{2}),\text{ where }A_{1}, A_{2}\in \mathscr{P}(\mathbb{R}^{m}).$$
    By \eqref{Arm in Stand.}, we have 
     \begin{align*}
       \forall~ \alpha_{k}\in (0,\bar{\alpha}_k] \text{ and } k \ge 0: F(x_{k}+\alpha_{k} d_{k})\preceq^{{\ell}}\left\{f^{a_{k,j}}(x_{k})+\rho \alpha_{k}\mathcal{F}^{k}(x_{k},d_{k})e\right\}_{j\in [\omega_{k}]}.
    \end{align*}
    As the function $\mathscr{G}$ is monotonic, 
    \begin{align*}
     \left(\mathscr{G}\circ F\right)(x_{k+1})~&\leq~ \underset{j\in [\omega_{k}]}{\min}\left\{\psi_{e}\left(f^{a_{k,j}}(x_{k})+\rho \alpha_{k}\mathcal{F}^{k}(x_{k},d_{k})e\right)\right\} \\
     ~&\leq~ \underset{j\in [\omega_{k}]}{\min}\left\{\psi_{e}\left(f^{a_{k,j}}(x_{k})\right)-\rho \alpha_{k}\mathcal{F}^{k}(x_{k},d_{k})\psi_{e}(-e)\right\}\\
     ~& =~ \underset{j\in [\omega_{k}]}{\min}\left\{\psi_{e}\left(f^{a_{k,j}}(x_{k})\right)+\rho \alpha_{k}\mathcal{F}^{k}(x_{k},d_{k})\right\}\\
     ~&=~ \left(\mathscr{G}\circ F\right)(x_{k})+\rho \alpha_{k}\mathcal{F}^{k}(x_{k},d_{k}).
    \end{align*}
    Applying this relation repeatedly, we get
    \begin{align*}
   ~&~\left(\mathscr{G}\circ F\right)(x_{k+1})-\left(\mathscr{G}\circ F\right)(x_{0})\leq \rho \sum_{i=0}^{k}\alpha_{i}~\mathcal{F}^{i}(x_{i},d_{i})\\
  \implies ~&~ \mathscr{C}\leq \rho \sum_{i=0}^{k}\alpha_{i}~\mathcal{F}^{i}(x_{i},d_{i})\leq 0, \text{ where }\mathscr{C}=\underset{s\in S}{\inf}\psi_e(s)-\left(\mathscr{G}\circ F\right)(x_{0})\\
 \implies  ~&~ \tfrac{\mathscr{C}}{\sigma-1}\geq \tfrac{\rho}{\sigma-1} \sum_{i=0}^{k}\alpha_{i}~\mathcal{F}^{i}(x_{i},d_{i})\geq 0.
    \end{align*}
    Therefore, it follows that $\sum_{k\geq 0}\alpha_{k}\frac{\mathcal{F}^{k}(x_{k},d_{k})}{\sigma-1}< +\infty,$ 
    and by the relation \eqref{ineq.for.zau}, we get 
    $$\sum_{k\geq 0}\frac{{\mathcal{F}^{k}}^{2}(x_{k},d_{k})}{\|d_{k}\|^{2}}< +\infty.$$
\end{proof}

\section{A general nonlinear conjugate gradient method}\label{section4}
In this section, we present a general nonlinear conjugate gradient scheme (Algorithm \ref{algo}) for the set optimization problem \eqref{SPL}. Depending on different ways of choosing $\beta_k$ in Step 4 of Algorithm \ref{algo}, we get different special schemes of the nonlinear conjugate gradient method for \eqref{SPL}. We explore later, in Section \ref{three_special_methods}, three tactful choices of $\beta_k$.

\begin{algorithm}[h!] 
\caption{A general nonlinear conjugate gradient method for the set optimization problem \eqref{SPL}}\label{algo}
\textbf{Step 0}. Choose an arbitrary initial point $x_{0}\in\mathbb{R}^{n}$, and the parameters $\rho\in (0,1)$ and $\sigma\in (\rho,1)$.  Provide a precision scalar $\varepsilon>0$. 
Set the iteration counter $k := 0.$ \\

\textbf{Step 1}. Find $$M_k:=\mathrm{Min}(F(x_k),K),~\omega_{k} := \lvert \mathrm{Min}(F(x_k),K)\rvert,\text{ and }P_{k}:=P_{x_{k}}.$$

\textbf{Step 2}. Compute
$$(a_{k},u_{k})\in \underset{(a,u)\in P_{k}\times\mathbb{R}^{n}}{\mathrm{argmin}} \left(\underset{j\in [\omega_{k}]}{\max}{\left\{\psi_{e}\left(\nabla f^{a_{j}}(x_{k})^{\top}u\right)\right\}}+ \tfrac{1}{2}\|u\|^{2}\right).$$

\textbf{Step 3}. If $\|u_{k}\| < \varepsilon$, then stop. Otherwise, go to Step 4.\\

\textbf{Step 4}. Find \begin{equation}\label{def_dk}
d_{k}:=\left\{\begin{aligned}
& u_{k}
& & \text{ if } k=0 \text{ or }\left\lvert\mathcal{F}^{k-1}(x_{k},d_{k-1})\right\rvert< \mathcal{F}^{k}(x_{k},d_{k-1}),\\
& u_{k}+\beta_{k}d_{k-1}
& & \text{ otherwise}, \\
\end{aligned}\right.   
\end{equation}
where $\beta_{k} $ is a given parameter.\\

\textbf{Step 5}. Find a step length $\alpha_{k}>0$ such that
\begin{align*}
   & \left\{\begin{aligned}
& f^{a_{j}}(x_{k}+\alpha_{k} d_{k})\preceq_{K}f^{a_{j}}(x_{k})+\rho \alpha_{k} \mathcal{F}^{k}(x_{k},d_{k})e ~\text{ for all } j\in [\omega_{k}]\\
& \mathcal{F}^{k}(x_{k}+\alpha_{k}d_{k},d_{k}) \geq \sigma\mathcal{F}^{k}(x_{k},d_{k}),\\
\end{aligned}\right.\\
 \noalign{\text{or} }\\
   & \left\{\begin{aligned}
& f^{a_{j}}(x_{k}+\alpha_{k} d_{k})\preceq_{K}f^{a_{j}}(x_{k})+\rho \alpha_{k} \mathcal{F}^{k}(x_{k},d_{k})e ~\text{ for all }~  j\in [\omega_{k}]\\
& \left\lvert \mathcal{F}^{k}(x_{k})+\alpha_{k}d_{k},d_{k})\right\rvert \leq \sigma|\mathcal{F}^{k}(x_{k},d_{k})|.\\
\end{aligned}\right. 
\end{align*}
   Set $x_{k+1} := x_{k}+\alpha_{k}d_{k},~ k\leftarrow k+1,$ and go to Step 1.
\end{algorithm}

\subsection{Well-definedness of Algorithm \ref{algo}}
The well-definedness of Algorithm \ref{algo} depends only on Steps 2 and 5. \\ 

For any iterate \( x_{k} \in \mathbb{R}^n \), Remark \ref{ex_a_u} (i) ensures the existence of an \( (a_{k}, u_{k}) \) for Step 2. So, Step 2 is well-defined. However, notice that the choice of $(a_{k},u_{k})$, in \emph{Step 2}, may not be unique. Even if there are multiple $(a_{k},u_{k})$'s, we are free to choose any $(a_{k},u_{k})$ because irrespective of which $(a_{k},u_{k})$ is chosen in Step 2, the corresponding $u_k$ is a $K$-descent direction of $F$ at $x_k$. During the proof of the global convergence of Algorithm \ref{algo} (Theorem ), we will see that we just require that $d_k$ is a $K$-descent direction.  \\

For Step 5 of Algorithm \ref{algo}, in Theorem \ref{exi.step}, we have proved a guarantee of the existence of a step length \( \alpha_{k} \) that satisfies the standard or the strong Wolfe conditions along a $K$-descent direction \( d_{k} \). So, Algorithm \ref{algo} is well-defined provided at $x_k$ the direction $d_k$ in Step 4 is a $K$-descent direction of the set-valued objective function $F$. In the next lemma, we provide some general choices for the parameter $\beta_k$ in Step 4 of Algorithm \ref{algo} so that $d_k$, chosen by the formula \eqref{def_dk}, is a $K$-descent direction of the objective function of the problem \eqref{SPL}. \\

\begin{proposition}\label{Pro_Bk_des_suf_des}

Suppose $\{x_k\}$ is a sequence of nonstationary points generated by \emph{Algorithm \ref{algo}} and for all $k = 0, 1, 2, \ldots,$ the value of $\beta_{k}$ in \emph{Step 4} of \emph{Algorithm \ref{algo}} is chosen by the rule 
\begin{align}\label{lem1.21}
        \beta_{k}\in \left \{\begin{aligned}
            & [0,\infty) & \text{ if }~ \mathcal{F}^{k}(x_{k},d_{k-1})\leq 0,\\
            & \left[0, -\tfrac{\mathcal{F}^{k}(x_{k},u_{k})}{\mathcal{F}^{k}(x_{k},d_{k-1})}\right] & \text{ if }~ \mathcal{F}^{k}(x_{k},d_{k-1})> 0,
        \end{aligned}\right.
    \end{align}
    or 
    \begin{align}\label{lem1.22}
    \beta_{k}\in \left \{\begin{aligned}
            & [0,\infty) & \text{ if }\mathcal{F}^{k}(x_{k},d_{k-1})\leq 0,\\
            & \left[0, -\mu \tfrac{\mathcal{F}^{k}(x_{k},u_{k})}{\mathcal{F}^{k}(x_{k},d_{k-1})}\right] & \text{ if }\mathcal{F}^{k}(x_{k},d_{k-1})> 0,
        \end{aligned}\right.          
    \end{align}
    where $\mu \in (0,1)$. If $\beta_{k}$ is chosen by \eqref{lem1.21},  then $d_{k}$ is a $K$-descent direction of $F$ at $x_k$. Moreover, if $\beta_{k}$ is chosen by the rule \eqref{lem1.22}, then the direction $d_{k}$ satisfies the sufficient descent condition \eqref{suf_d} with $c=1-\mu$.
\end{proposition}

\begin{proof}
We prove the result for the case of \eqref{lem1.22}. The case of \eqref{lem1.21} can be similarly proved. \\ \\ 
Let $\beta_k$ be chosen by the rule \eqref{lem1.22}. Then, clearly 
for $k=0$, as $d_{0} = u_{0}$, $d_{k}$ satisfies the sufficient descent condition \eqref{suf_d} since 
$$\mathcal{F}^{k}(x_{k},d_{k})=\mathcal{F}^{k}(x_{k},u_{k})\leq (1-\mu)\mathcal{F}^{k}(x_{k},u_{k}).$$

\noindent Next, assume that $k\geq 1$. From the definition of $d_{k}$ and Lemma \ref{pro_Ger.} (ii), we obtain
\begin{equation}\label{lem1.24}
\psi_{e}\left(\nabla f^{a_{k,j}}(x_{k})^{\top}d_{k}\right)\leq \psi_{e}\left(\nabla f^{a_{k,j}}(x_{k})^{\top}u_{k}\right)+\beta_{k}\psi_{e}\left(\nabla f^{a_{k,j}}(x_{k})^{\top}d_{k-1}\right) \text{ for all } j\in [\omega_{k}]
    \end{equation}
because $\beta_{k}\geq 0$. Thus, for the element $a_{k} \in P_{k}$ at $x_{k}$, by using the relation \eqref{lem1.24}, we get 
\begin{equation}\label{lem1.23}
    \mathcal{F}^{k}(x_{k},d_{k})\leq \mathcal{F}^{k}(x_{k},u_{k})+\beta_{k}\mathcal{F}^{k}(x_{k},d_{k-1}).
\end{equation}
Since $\mathcal{F}^{k}(x_{k},u_{k})< 0$, if $\mathcal{F}^{k}(x_{k},d_{k-1})\leq 0$, then from \eqref{lem1.23} we obtain 
\begin{equation*}
    \mathcal{F}^{k}(x_{k},d_{k})\leq \mathcal{F}^{k}(x_{k},u_{k})\leq (1-\mu)\mathcal{F}^{k}(x_{k},u_{k}).
\end{equation*}
Now, suppose that $\mathcal{F}^{k}(x_{k},d_{k-1})> 0$ and $\beta_{k}\in \left[0, -\mu \tfrac{\mathcal{F}^{k}(x_{k},u_{k})}{\mathcal{F}^{k}(x_{k},d_{k-1})}\right]$. Then,
\begin{equation*}
 \mathcal{F}^{k}(x_{k},d_{k})\leq \mathcal{F}^{k}(x_{k},u_{k})-\mu \mathcal{F}^{k}(x_{k},u_{k})=(1-\mu)\mathcal{F}^{k}(x_{k},u_{k}).   
\end{equation*}
Thus, $d_{k}$ satisfies the sufficient descent condition \eqref{suf_d} with $c=1-\mu$ when $\beta_{k}$ is chosen by the rule \eqref{lem1.22}. 
\end{proof}

\begin{remark}\label{Remark on Algo.}
In \emph{Step 4} of \emph{Algorithm \ref{algo}}, note that we use a \emph{regular restart}   
$d_k = u_k$ when $\left\lvert\mathcal{F}^{k-1}(x_{k},d_{k-1})\right\rvert< \mathcal{F}^{k}(x_{k},d_{k-1})$.  
The reason behind using this condition for the regular restart $d_k = u_k$ is that  although under the condition $\left\lvert\mathcal{F}^{k-1}(x_{k},d_{k-1})\right\rvert \ge \mathcal{F}^{k}(x_{k},d_{k-1})$ the direction $u_k + \beta_k d_{k - 1}$ is a $K$-descent direction of $F$ at $x_k$ $($see the lines of the proof of \emph{Theorem \ref{Conv_Algo}}$)$, we could not find any guarantee of $K$-decency of $u_k + \beta_k d_{k - 1}$ if $\left\lvert\mathcal{F}^{k-1}(x_{k},d_{k-1})\right\rvert \ge \mathcal{F}^{k}(x_{k},d_{k-1})$. So, we have used the known descent direction $u_k$ as $d_k$ when $\left\lvert\mathcal{F}^{k-1}(x_{k},d_{k-1})\right\rvert < \mathcal{F}^{k}(x_{k},d_{k-1})$. It is noteworthy that in conjugate direction methods for vector optimization problems \emph{\cite{lucambio2018nonlinear}}, this regular restart condition is not required because for all $k=0,1,\ldots$, $\mathcal{F}^{k-1}(x_k,d)=\mathcal{F}^{k}(x_k,d)$ for any $d \in \mathbb{R}^n$. 
\end{remark}

\subsection{Convergence analysis}\label{Conver_Analysis}
In this section, we examine the global convergence property of Algorithm \ref{algo}. It is important to note that the following global convergence result holds for any choice of the conjugate gradient parameter $\beta_k \ge 0$ for all $k = 0, 1, 2, \ldots$.  \\

\begin{theorem}\label{Conv_Algo}
    Let \emph{Assumptions \ref{lip_ass}} and \emph{\ref{bddbe_ass}} hold  and $\{x_k\}$ be a sequence of nonstationary points generated by \emph{Algorithm \ref{algo}}. Furthermore, assume that 
    
    \begin{equation}\label{the1.3e6}
        \sum_{k\geq 0}{\frac{1}{\|d_{k}\|^{2}}}=+\infty.
    \end{equation}
    
    \begin{enumerate}[label=(\roman*)]
        \item If $\alpha_{k}$ satisfies the standard Wolfe conditions \eqref{Arm with stan. wol.} and $d_{k}$ satisfies the sufficient descent condition \eqref{suf_d} at $x_{k}$ for all $k = 0, 1, 2, \ldots$, then $\underset{k\rightarrow \infty}{\liminf}~\|u_{k}\|=0$.

        \item If $\alpha_{k}$ satisfies the strong Wolfe conditions \eqref{Arm with strong. wol.}, $\beta_{k}\geq 0$, and $d_{k}$ is a $K$-descent direction of $F$ at $x_{k}$ for all $k = 0, 1, 2, \ldots$, then $\underset{k\rightarrow \infty}{\liminf}~\|u_{k}\|=0$.
    \end{enumerate}
\end{theorem}

\begin{proof}
On the contrary, suppose that (i) is not true. Then, we can find a constant $\gamma>0$ such that
\begin{equation}\label{the.1.3e11}
  \|u_{k}\|\geq \gamma \text{ for all } k = 0, 1, 2, \ldots. 
\end{equation}
Notice that $x_{k}$ is a nonstationary point of \eqref{SPL}. Therefore, by Proposition \ref{Pro. station. cond.} and the sufficient descent condition \eqref{suf_d}, we have
\begin{align}\label{the1.3e1}
    & \mathcal{F}^{k}(x_{k},u_{k})+\tfrac{1}{2}\|u_{k}\|^{2} < 0, 
 \end{align}
 and for some $c>0$, 
\begin{align}\label{the1.3e2}
   \mathcal{F}^{k}(x_{k}, d_{k})\leq c~\mathcal{F}^{k}(x_{k},u_{k}) < 0.  
\end{align}

\noindent From relations \eqref{the.1.3e11}, \eqref{the1.3e1}, and \eqref{the1.3e2}, we have for all $k = 0, 1, 2, \ldots$ that 
\begin{equation*}
0< \tfrac{c^{2}\gamma^{4}}{4\|d_{k}\|^{2}}\leq \tfrac{c^{2}\|u_{k}\|^{4}}{4\|d_{k}\|^{2}}\leq \tfrac{c^{2}{\mathcal{F}^{k}}^{2}(x_{k},u_{k})}{\|d_{k}\|^{2}}\leq \tfrac{{\mathcal{F}^k}^{2}(x_{k},d_{k})}{\|d_{k}\|^{2}}. 
\end{equation*}

\noindent Thus, from \eqref{the1.3e6}, we get 
\begin{equation*}
    \sum_{k\geq 0}{\tfrac{{\mathcal{F}^{k}}^{2}(x_{k},d_{k})}{\|d_{k}\|^{2}}}= +\infty.
\end{equation*}
However, this is contradictory to Theorem \ref{the. zoun.}. Thus, the statement (i) is true. \\

For the proof of the statement (ii), similar to (i), if it is false, there exists a constant $\gamma>0$ such that the inequality \eqref{the.1.3e11} holds. \\  \\
Due to \eqref{def_dk}, for $k=0$ and $\left\lvert\mathcal{F}^{k-1}(x_{k},d_{k-1})\right\rvert < \mathcal{F}^{k}(x_{k},d_{k-1})$, we get

\[d_{k}=u_{k}\implies \mathcal{F}^{k}(x_{k},d_{k})=\mathcal{F}^{k}(x_{k},u_{k}).\]

\smallskip
\smallskip
\smallskip
\noindent In this scenario, $d_{k}$ satisfies the sufficient descent condition \eqref{suf_d}, and the step-length $\alpha_{k}$ satisfies the standard Wolfe condition \eqref{stan, wol.} as it is satisfying the strong Wolfe condition \eqref{strong. wol.}. Therefore, from statement (i), we obtain

 \[\underset{k\rightarrow \infty}{\liminf} ~ \|u_{k}\|=0.\]

\smallskip
\smallskip
\smallskip
\noindent Now, we assume that $k\geq 1$ and $\left\lvert\mathcal{F}^{k-1}(x_{k},d_{k-1})\right\rvert \ge \mathcal{F}^{k}(x_{k},d_{k-1})$. From the definition of $d_{k}$ in \eqref{def_dk}, we have
\begin{align}
   & d_{k}=u_{k}+\beta_{k}d_{k-1}\nonumber\\
    \implies &\mathcal{F}^{k}(x_{k},d_{k}) \leq \mathcal{F}^{k}(x_{k},u_{k})+\beta_{k}\mathcal{F}^{k}(x_{k},d_{k-1})\label{the1.3e3}
\end{align}
because $\beta_{k}\geq 0$. 
Since $u_{k}$ is a $K$-descent direction and $\alpha_{k}$ satisfies the strong Wolfe condition \eqref{strong. wol.}, we obtain from \eqref{the1.3e3} that 
\begin{align}
    0\leq -\mathcal{F}^{k}(x_{k},u_{k}) & \leq - \mathcal{F}^{k}(x_{k},d_{k})+\beta_{k}\mathcal{F}^{k}(x_{k},d_{k-1})\nonumber \\
    & \leq - \mathcal{F}^{k}(x_{k},d_{k}) - \sigma \beta_{k}\mathcal{F}^{k-1}(x_{k-1},d_{k-1}) \text{ as }\left\lvert\mathcal{F}^{k-1}(x_{k},d_{k-1})\right\rvert > \mathcal{F}^{k}(x_{k},d_{k-1})\nonumber\\
    \implies {\mathcal{F}^{k}}^{2}(x_{k},u_{k}) &\leq {\mathcal{F}^{k}}^{2}(x_{k},d_{k})+\sigma^{2} \beta^{2}_{k}{\mathcal{F}^{k-1}}^{2}(x_{k-1},d_{k-1})+2\sigma \beta_{k} \mathcal{F}^{k}(x_{k},d_{k}) \mathcal{F}^{k-1}(x_{k-1},d_{k-1})\nonumber\\
    & \leq (1+2\sigma^{2})\left[{\mathcal{F}^{k}}^{2}(x_{k},d_{k})+\tfrac{\beta^{2}_{k}}{2}{\mathcal{F}^{k-1}}^{2}(x_{k-1},d_{k-1})\right]\text{ from Lemma }\ref{S_calcul.}~ (ii)\nonumber\\
    \overset{\eqref{the1.3e1}}{\implies} c_{1}\|u_{k}\|^{4}  \leq ~&{\mathcal{F}^{k}}^{2}(x_{k},d_{k})+\tfrac{\beta^{2}_{k}}{2}{\mathcal{F}^{k-1}}^{2}(x_{k-1},d_{k-1}), \label{the1.3e4}
\end{align}
where $c_{1}=\tfrac{1}{(1+2\sigma^{2})}$.\\ \\ 
Using the definition of $d_{k}$  in \eqref{def_dk} and Lemma \ref{S_calcul.} (iii), we have
\begin{align}
    & -\beta_{k}d_{k-1}=-d_{k}+u_{k}\nonumber\\
    \implies&\beta^{2}_{k}\|d_{k-1}\|^{2}\leq \left[\|d_{k}\|+\|u_{k}\|\right]^{2}\leq 2\|d_{k}\|^{2}+2\|u_{k}\|^{2}\nonumber\\
    \implies &\|d_{k}\|^{2} \geq -\|u_{k}\|^{2}+\tfrac{\beta^{2}_{k}}{2}\|d_{k-1}\|^{2}.\label{the1.3e5}
\end{align}
Note that
\begin{align*}
    \tfrac{{\mathcal{F}^{k}}^{2}(x_{k},d_{k})}{\|d_{k}\|^{2}}+\tfrac{{\mathcal{F}^{k-1}}^{2}(x_{k-1},d_{k-1})}{\|d_{k-1}\|^{2}}&=\tfrac{1}{\|d_{k}\|^{2}}\left[{\mathcal{F}^{k}}^{2}(x_{k},d_{k})+\tfrac{\|d_{k}\|^{2}}{\|d_{k-1}\|^{2}}{\mathcal{F}^{k-1}}^{2}(x_{k-1},d_{k-1}) \right]\\
    & \overset{\eqref{the1.3e5}}{\geq} \tfrac{1}{\|d_{k}\|^{2}}\left[{\mathcal{F}^{k}}^{2}(x_{k},d_{k})+\left(\tfrac{\beta^{2}_{k}}{2}-\tfrac{\|u_{k}\|^{2}}{\|d_{k-1}\|^{2}}\right){\mathcal{F}^{k-1}}^{2}(x_{k-1},d_{k-1}) \right]\\
    & \overset{\eqref{the1.3e4}}{\geq} \tfrac{\|u_{k}\|^{2}}{\|d_{k}\|^{2}}\left[c_{1}\|u_{k}\|^{2}-\tfrac{{\mathcal{F}^{k-1}}^{2}(x_{k-1},d_{k-1})}{\|d_{k-1}\|^{2}}\right].
    \end{align*}
   From Theorem \ref{the. zoun.}, $\tfrac{\mathcal{F}^{k}(x_{k},d_{k})}{\|d_{k}\|}\rightarrow 0.$ Therefore,
   \begin{align*}
    \tfrac{{\mathcal{F}^{k}}^{2}(x_{k},d_{k})}{\|d_{k}\|^{2}}+\tfrac{{\mathcal{F}^{k-1}}^{2}(x_{k-1},d_{k-1})}{\|d_{k-1}\|^{2}}&\geq\tfrac{c_{1}}{2} \tfrac{\|u_{k}\|^{4}}{\|d_{k}\|^{2}} \geq\tfrac{c_{1}\gamma^{4}}{2} \tfrac{1}{\|d_{k}\|^{2}}   
\end{align*}

\noindent for sufficiently large $k$. This inequality, together with  Theorem \ref{the. zoun.}, yields $\sum_{k \ge 0} \tfrac{1}{\|d_{k}\|^{2}} < + \infty$, which is contradictory to \eqref{the1.3e6}. Thus, the result follows. 

\end{proof}

\section{Three special conjugate gradient methods}\label{three_special_methods}
In this section, we explore three special choices of the conjugate gradient parameter $\beta_k$ in Step 4 of the general nonlinear conjugate gradient Algorithm \ref{algo}. \\ 

\begin{itemize}
    \item Dai-Yuan (DY):
    \begin{align}\label{DY_paramete}
    \beta_{\mathrm{DY}_k} := \frac{-\mathcal{F}^{k}(x_k,u_{k})}{\mathcal{F}^{k-1}(x_k,d_{k-1}) - \mathcal{F}^{k-1}(x_{k-1}, d_{k-1})}.
\end{align}

\item Polak-Ribi\`{e}re-Polyak (PRP): 
\begin{align}\label{PRP_paramete}
     \beta_{\mathrm{PRP}_k} := \frac{-\mathcal{F}^{k}(x_k,u_{k})+\mathcal{F}^{k}(x_{k-1},u_{k})}{- \mathcal{F}^{k-1}(x_{k-1}, u_{k-1})}. 
\end{align}

\item Hastenes-Stiefel (HS): 
  \begin{align}\label{HS_paramete}
      \beta_{\mathrm{HS}_k} := \frac{-\mathcal{F}^{k}(x_k,u_{k})+\mathcal{F}^{k}(x_{k-1},u_{k})}{\mathcal{F}^{k-1}(x_k,d_{k-1}) - \mathcal{F}^{k-1}(x_{k-1}, d_{k-1})}.
  \end{align}
    
\end{itemize}

Next, we discuss the global convergence of Algorithm \ref{algo} for these three particular choices of $\beta_{k}$ in Step 4 of Algorithm \ref{algo}. Before analyzing the convergence with the parameter $\beta_{\mathrm{DY}_k}$, we demonstrate that the direction $d_k$ generated by Algorithm \ref{algo} with the parameter $\beta_{\mathrm{DY}_k}$ satisfies the sufficient descent condition \eqref{suf_d}. This result is crucial in the subsequent proof of global convergence.\\

\begin{lemma}\label{DY_s.D}
 Consider \emph{Algorithm \ref{algo}} with \(0 \leq \beta_k \leq \beta_{\mathrm{DY}_k}\) for all $k = 0, 1, 2, \ldots$. Assume that $\{x_{k}\}$ is a sequence of nonstationary points generated by this algorithm when $\{\alpha_{k}\}$ is chosen by the strong Wolfe conditions \eqref{Arm with strong. wol.}. Then, \(d_k\) satisfies sufficient descent condition \eqref{suf_d} with \(c = \tfrac{1}{1 + \sigma}\) for all $k = 0, 1, 2, \ldots$. 
\end{lemma}

\begin{proof}
 We prove this lemma by the method of induction. \\ \\  
Due to the definition  \eqref{def_dk} of $d_{k}$, for $k = 0$, $d_0 = u_{0}$. Since $\mathcal{F}_{a_0}(x_0, u_{0}) < 0$ and $0 < \sigma < 1$, we have 
\[\mathcal{F}_{a_{0}}(x_{0}, d_{0})\leq \tfrac{1}{1+\sigma}\mathcal{F}_{a_{0}}(x_{0},u_{0}),\]
i.e.,  \eqref{suf_d} holds with $c = \tfrac{1}{1 + \sigma}$.\\ \\ 
Assume that for some $k \geq 1$, the following condition holds: $$\mathcal{F}^{k-1}(x_{k-1}, d_{k-1})\leq \tfrac{1}{1+\sigma}\mathcal{F}^{k-1}(x_{k-1},u_{k-1}).$$
Observe that if $\left\lvert\mathcal{F}^{k-1}(x_{k},d_{k-1})\right\rvert < \mathcal{F}^{k}(x_{k},d_{k-1})$, then
\[d_{k}=u_{k}\implies \mathcal{F}^{k}(x_{k},d_{k})=\mathcal{F}^{k}(x_{k},u_{k}).\]
Accordingly, the direction $d_{k}$  holds the condition \eqref{suf_d} with $c = \tfrac{1}{1 + \sigma}$.
\noindent Therefore, we aim to show that for any $k\in \mathbb{N}$ with $\left\lvert\mathcal{F}^{k-1}(x_{k},d_{k-1})\right\rvert > \mathcal{F}^{k}(x_{k},d_{k-1})$, $d_{k}$ satisfies the condition \eqref{suf_d}. Since $\alpha_{k}$ satisfies the strong Wolfe condition \eqref{strong. wol.}, we obtain
\begin{align*}
  \mathcal{F}^{k-1}(x_{k}, d_{k-1})\geq \sigma \mathcal{F}^{k-1}\left(x_{k-1},d_{k-1}\right)>\mathcal{F}^{k-1}\left(x_{k-1},d_{k-1}\right)
   \end{align*}
  because $0<\sigma <1$ and $\mathcal{F}^{k-1}\left(x_{k-1},d_{k-1}\right)<0$.
  Therefore,
   \begin{align*}
   \mathcal{F}^{k-1}(x_k,d_{k-1})-\mathcal{F}^{k-1}\left(x_{k-1},d_{k-1}\right)> 0.
\end{align*}
Thus, the denominator of $\beta_{\mathrm{DY}_k}$ as given in \eqref{DY_paramete} is positive. Consequently, $\beta_{\mathrm{DY}_k}$ is well-defined. Since $\mathcal{F}^{k}(x_k,u_{k})<0$, we have  $\beta_{\mathrm{DY}_k}>0$.\\ \\ 
From the definition of $d_{k}$, for $k\geq1$, we have
$d_{k}=u_{k}+\beta_{k}d_{k-1}.$ Therefore,
\[ \nabla f^{a_{k,j}}(x_{k})^{\top}d_{k}= \nabla f^{a_{k,j}}(x_{k})^{\top}u_{k}+\beta_{k}\nabla f^{a_{k,j}}(x_{k})^{\top}d_{k-1},   j\in [\omega_{k}].
\]
From Lemma \ref{pro_Ger.} (i), for all $j\in [\omega_{k}]$, it follows that 
\begin{align*}
&\psi_{e}\left(\nabla f^{a_{k,j}}(x_{k})^{\top}d_{k}\right)\leq \psi_{e}\left(\nabla f^{a_{k,j}}(x_{k})^{\top}u_{k}\right)+\beta_{k}\psi_{e}\left(\nabla f^{a_{k,j}}(x_{k})^{\top}d_{k-1}\right)\\
\implies & \mathcal{F}^{k}(x_{k}, d_{k})\leq \mathcal{F}^{k}(x_{k}, u_{k})+\beta_{k} \mathcal{F}^{k}(x_{k}, d_{k-1}).
\end{align*}
Notice that if $\mathcal{F}^{k}(x_{k}, d_{k-1})\leq 0$, then $d_k$ satisfies the sufficient descent condition \eqref{suf_d}. Therefore, assume that $\mathcal{F}^{k}(x_{k}, d_{k-1})> 0$ and $\beta_{\mathrm{DY}_k}>0$. Then, we have
\begin{align*}
  \mathcal{F}^{k}(x_{k}, d_{k})&\leq \mathcal{F}^{k}(x_{k}, u_{k})+\beta_{\mathrm{DY}_k} \mathcal{F}^{k}(x_{k}, d_{k-1}) \text{ since }0<\beta_{k}\leq \beta_{\mathrm{DY}_k}\\
 &\leq \tfrac{-\mathcal{F}^{k-1}(x_{k-1}, d_{k-1})\mathcal{F}^{k}(x_{k}, u_{k})}{\mathcal{F}^{k-1}(x_k,d_{k-1})-\mathcal{F}^{k-1}\left(x_{k-1},d_{k-1}\right)}\\
 &\leq \tfrac{\mathcal{F}^{k}(x_{k}, u_{k})}{1-\mathcal{V}_{k}}, \text{ where }\mathcal{V}_{k}:=\tfrac{\mathcal{F}^{k-1}(x_k,d_{k-1})}{\mathcal{F}^{k-1}\left(x_{k-1},d_{k-1}\right)}\\
 & \leq \tfrac{\mathcal{F}^{k}(x_{k}, u_{k})}{1+\sigma} \text{ because from }\eqref{strong. wol.}, ~\mathcal{V}_{k}\in [-\sigma,\sigma],
\end{align*}
which completes the proof.
\end{proof}

\smallskip

In the next theorem, under a suitable hypothesis, we demonstrate that global convergence can be achieved if $\beta_k$ is an appropriate fraction of the DY parameter $\beta_{\mathrm{DY}_k}$. \\

\begin{theorem}\label{Con_DY}
    Consider \emph{Algorithm \ref{algo}} with $\beta_k = \eta \beta_{\mathrm{DY}_k}$ for all $k = 0, 1, 2, \ldots$, where $0 \leq \eta < \tfrac{1 - \sigma}{1 + \sigma}$. Suppose that \emph{Assumptions \ref{lip_ass}} and \emph{\ref{bddbe_ass} }are satisfied. If $\{x_{k}\}$ is a sequence of nonstationary points generated by this algorithm where $\alpha_k$ is chosen by the strong Wolfe conditions \eqref{Arm with strong. wol.} for all $k = 0, 1, 2, \ldots$, then $\underset{k \to \infty}{\liminf} ~ \| u_{k} \| = 0$.
\end{theorem}

\begin{proof}
By the strong Wolfe condition \eqref{strong. wol.}, we obtain
\begin{align}\label{conver_1}
  \mathcal{F}^{k-1}(x_k,d_{k-1}) - \mathcal{F}^{k-1}(x_{k-1}, d_{k-1})\geq (\sigma - 1) \mathcal{F}^{k-1}(x_{k-1}, d_{k-1}).  
\end{align}
 According to Lemma \ref{DY_s.D}, $d_k$ satisfies the sufficient descent direction \eqref{suf_d} with $c = \tfrac{1}{1 + \sigma}$ for all $k \geq 0$. Therefore, from \eqref{conver_1}, we get
 $$\mathcal{F}^{k-1}(x_k,d_{k-1}) - \mathcal{F}^{k-1}(x_{k-1}, d_{k-1})\geq \tfrac{\sigma - 1}{1 + \sigma} \mathcal{F}^{k-1}(x_{k-1},u_{k-1}) > 0$$
 because $0<\sigma<1$ and $\mathcal{F}^{k-1}(x_{k-1},u_{k-1})<0$.
Thus, we obtain
\begin{align*}
\tfrac{-\mathcal{F}^{k-1}(x_{k-1},u_{k-1})}{\mathcal{F}^{k-1}(x_k,d_{k-1}) - \mathcal{F}^{k-1}(x_{k-1}, d_{k-1})} \leq \tfrac{1 + \sigma}{1 - \sigma}.
\end{align*}
Define $\delta := \eta\tfrac{ 1 + \sigma}{1 - \sigma} < 1$. Thus, using the definition of $\beta_k$, we have
\begin{align*}
\beta_k & = \delta \tfrac{1 - \sigma}{1 + \sigma} \left[ \tfrac{-\mathcal{F}^{k}(x_k,u_{k})}{\mathcal{F}^{k-1}(x_k,d_{k-1}) -\mathcal{F}^{k}(x_{k-1}, d_{k-1})}  \right]\\
&= \delta \tfrac{1 - \sigma}{1 + \sigma} \left[\tfrac{-\mathcal{F}^{k}(x_k,u_{k})}{-\mathcal{F}^{k}(x_{k-1},u_{k-1})}\right] \left[\tfrac{-\mathcal{F}^{k}(x_{k-1},u_{k-1})}{\mathcal{F}^{k-1}(x_k,d_{k-1}) -\mathcal{F}^{k}(x_{k-1}, d_{k-1})}\right]\\
&\leq \delta \left[\tfrac{\mathcal{F}^{k}(x_k,u_{k})}{\mathcal{F}^{k}(x_{k-1},u_{k-1})}\right].\label{conv_2}    
\end{align*}
Assume on contrary that there exists $\gamma>0$ such that 
\begin{align}
    \|u_{k}\|\geq \gamma \text{ for all } k \ge 0. 
\end{align}
Note that 
\begin{align*}
    \|d_{k}\|^{2}\leq \bigl[\|u_{k}\|+\lvert \beta_k\rvert \|d_{k-1}\|\bigr]^{2}\leq \tfrac{1}{1-\delta^{2}}\|u_{k}\|^{2}+\tfrac{1}{\delta^{2}}\beta^{2}_{k}\|d_{k-1}\|^{2}
\end{align*}
because of Lemma \ref{S_calcul.} (iv) with $z_{1}=\|u_{k}\|,~ z_{2}=\lvert \beta_k\rvert \|d_{k-1}\|,$ and $t=\tfrac{\delta}{\sqrt{(1-\delta^{2})}}.$\\
Thus,
\begin{align*}
    \tfrac{\|d_{k}\|^{2}}{{\mathcal{F}^{k}}^{2}(x_{k},u_k)} &\leq \tfrac{1}{1-\delta^{2}}\tfrac{\|u_{k}\|^{2}}{{\mathcal{F}^{k}}^{2}(x_{k},u_k)}+\tfrac{1}{\delta^{2}}\tfrac{\beta^{2}_{k}\|d_{k-1}\|^{2}}{{\mathcal{F}^{k}}^{2}(x_{k},u_k)}\\
    & \overset{\eqref{conv_2}}{\leq} \tfrac{1}{1-\delta^{2}}\tfrac{\|u_{k}\|^{2}}{{\mathcal{F}^{k}}^{2}(x_{k},u_k)}+\tfrac{\|d_{k-1}\|^{2}}{{\mathcal{F}^{k-1}}^{2}(x_{{k-1}},u_{k-1})}.\\
\end{align*}
Since $0\leq \gamma^{2}\leq \|u_{k}\|^{2}<-2\mathcal{F}^{k}(x_{k},u_k)$, therefore
\begin{align*}
  \tfrac{\|d_{k}\|^{2}}{{\mathcal{F}^{k}}^{2}(x_{k},u_k)}\leq \tfrac{4}{(1-\delta^{2})\gamma^{2}}+\tfrac{\|d_{k-1}\|^{2}}{{\mathcal{F}^{k-1}}^{2}(x_{{k-1}},u_{k-1})}.   
\end{align*}
Applying this relation repeatedly, we obtain
\begin{align*}
  \tfrac{\|d_{k}\|^{2}}{{\mathcal{F}^{k}}^{2}(x_{k},u_k)}&\leq \tfrac{4}{(1-\delta^{2})\gamma^{2}}k+\tfrac{\|d_{0}\|^{2}}{{\mathcal{F}^{0}}^{2}(x_{0},u_{0})} \leq \tfrac{4}{(1-\delta^{2})\gamma^{2}}k+\tfrac{4}{\gamma^{2}}.
\end{align*}
Thus,
\begin{align}\label{r. below}
    \tfrac{{\mathcal{F}^{k}}^{2}(x_{k},u_k)}{\|d_{k}\|^{2}}\geq \tfrac{(1-\delta^{2})\gamma^{2}}{4(k+1-\delta^{2})}\geq \tfrac{(1-\delta^{2})\gamma^{2}}{4(k+1)}.
\end{align}
Now, by using the sufficient descent condition \eqref{suf_d} and the relation \eqref{r. below}, we have
\begin{align*}
    \sum_{k\geq 0}\tfrac{{\mathcal{F}^{k}}^{2}(x_{k},d_k)}{\|d_{k}\|^{2}}\geq   \sum_{k\geq 0}c^{2}\tfrac{{\mathcal{F}^{k}}^{2}(x_{k},u_{k})}{\|d_{k}\|^{2}}\geq \tfrac{c^{2}(1-\delta^{2})\gamma^{2}}{4}\sum_{k\geq 0}\tfrac{1}{k+1}=+\infty,
\end{align*}
which is contradictory to \eqref{Zounten.}. Thus, the result follows.

\end{proof}

Next, we consider the convergence analysis of Algorithm \ref{algo} for PRP and HS parameters. Our results are based on the work of Prudente et al.  \cite{lucambio2018nonlinear} for vector optimization and the work of Gilbert and Nocedal \cite{gilbert1992global} for scalar optimization by introducing the Property $(\ast)$. The vector extension of this property is given in \cite{lucambio2018nonlinear} as follows. \\

\begin{framed}
\noindent
\textbf{Property} $(\ast)$  Consider a nonlinear conjugate gradient scheme that satisfies 
\begin{equation}\label{aux_star}
  0<\gamma \leq \|u_{k}\|\leq \bar{\gamma} \text{ for all } k = 0, 1, 2, \ldots.   
\end{equation}
Under this setup, we say that an iterative method possesses Property $(\ast)$ if there exist constants $b>1$ and $\lambda>0$ such that for all $k = 0, 1, 2, \ldots$, we have $\beta_{k} \leq b$, and 
\begin{equation*}
      \|s_{k-1}\|\leq \lambda  \implies \beta_{k}\leq \tfrac{1}{2b}, 
\end{equation*}
where $s_{k-1}=x_{k}-x_{k-1}$.

\end{framed}

\vspace{0.5cm}

Now, we analyze whether the proposed method, possessing Property $(\ast)$, can lead the iterative points to converge to a stationary point of the set-valued objective function $F$. Following this, we demonstrate that PRP and HS methods exhibit Property $(\ast)$ under certain mild assumptions.\\

\begin{theorem}\label{theo req of prp con}
Let $\{x_{k}\}$ be a sequence of nonstationary points generated by \emph{Algorithm \ref{algo}} with $\beta_k \geq 0 $ for all $k \ge 0$. Suppose that \emph{Assumptions \ref{Ass. 3}} and \emph{\ref{lip_ass}} hold, and for all $k \ge 0$, $d_k$ satisfies the sufficient descent condition \eqref{suf_d} and $\alpha_k$  satisfies the standard Wolfe conditions \eqref{Arm with stan. wol.} at $x_k$. Moreover, we assume that the method exhibits \emph{Property} $(\ast)$. Then, $\underset{k \to \infty}{\liminf} \| u_{k} \| = 0 $.
\end{theorem}

\begin{proof}
The lines of the proof are similar to \cite[Theorem 5.10]{lucambio2018nonlinear}.
\end{proof}

We can now establish the convergence of the PRP and HS parameters. \\

\begin{theorem}
Let $\{x_k\}$ be a sequence of nonstationary points generated by Algorithm \emph{\ref{algo}} with $\beta_{k}=\max \{\beta_{\mathrm{PRP}_{k}},0\}$ or $\beta_{k}=\max \{\beta_{\mathrm{HS}_{k}},0\}$. Assume that \emph{Assumptions \ref{Ass. 3}} and \emph{\ref{lip_ass}} hold and for all $k \ge 0$, $d_{k}$ satisfies the sufficient descent condition \eqref{suf_d} and $\alpha_{k}$ satisfies the standard Wolfe conditions \eqref{Arm with stan. wol.} at $x_{k}$. Then, $\underset{k\rightarrow \infty}{\liminf}\|u_{k}\|=0$. 
\end{theorem}

\begin{proof}
In view of Theorem \ref{theo req of prp con}, Algorithm \ref{algo} holds the result with with $\beta_{k}=\max \{\beta_{\mathrm{PRP}_{k}},0\}$ or $\beta_{k}=\max \{\beta_{\mathrm{HS}_{k}},0\}$ if $\mathrm{PRP}$ and $\mathrm{HS}$ methods have Property $(\ast)$. To prove that these two choices of $\beta_k$ satisfy Property $(\ast)$, we begin with the assumption of \eqref{aux_star}. Then,  
\begin{align}\label{the1.511}
\mathcal{F}^{k}(x_k,u_{k})< -\tfrac{1}{2}\|u_{k}\|^{2}\leq -\tfrac{1}{2}\gamma^{2}.
\end{align}
Note that, from Lemma \ref{pro_Ger.} (iii), $\psi_{e}$ is a Lipschitz continuous function. Therefore, there exists a Lipschitz constant, say, $\bar{L}$ such that
\begin{align}\label{the1.51}
  \forall~{j}\in [\omega_{k}]:~  \left\lvert \psi_{e}\left(\nabla f^{a_{k,j}}(x_{k})^{\top} u_{k}\right) \right\rvert \leq \bar{L}\left\|\nabla f^{a_{k,j}}(x_{k})^{\top} u_{k}\right\| \leq \bar{L}\left\|\nabla f^{a_{k,j}}(x_{k})\right\| \| u_{k}\|.
\end{align}
From the definition of $\mathcal{F}^{k}$, we have
\begin{align}
     \mathcal{F}^{k}(x_k,u_{k}) &=  \psi_{e}\left(\nabla f^{a_{k,j_{k}}}(x_{k})^{\top} u_{k}\right), \text{ where }j_{k}=\underset{j\in [\omega_{k}]}{\text{argmax}}\left\{\psi_{e}\left(\nabla f^{a_{k,j}}(x_{k})^{\top}u_{k}\right)\right\} \nonumber\\
      &= -\left \lvert \psi_{e}\left(\nabla f^{a_{k,j_{k}}}(x_{k})^{\top} u_{k}\right)\right  \rvert \overset{\eqref{the1.51}}{\geq} -\bar{L} \left \|\nabla f^{a_{k,j_{k}}}(x_{k})\|\|u_{k} \right\| \nonumber\\
    \implies  -\mathcal{F}^{k}(x_k,u_{k})&\leq \bar{L}\left\|\nabla f^{a_{k,j_{k}}}(x_{k})\right\|\|u_{k}\|. \label{the1.52}
\end{align}
From the continuity arguments, there exists constant $\tilde{c}>0$ such that $\left\|\nabla f^{a_{k,j}}(x_{k})\right\|\leq \tilde{c}$ for all $k\geq 0$ and $j\in [\omega_{k}]$ since $\{x_{k}\}\subset \mathcal{L}$. Furthermore, assume that $\|u_{k}\|\leq \tilde{\gamma} \text{ for all }k\geq 0$. Therefore, using \eqref{the1.511} and \eqref{the1.52}, we get
\begin{align*}
     \tfrac{1}{2}\gamma^{2}\leq    -\mathcal{F}^{k}(x_k,u_{k})\leq  \bar{\gamma}_{1} := \bar{L}\tilde{c}\tilde{\gamma}.
\end{align*}
Similarly, we can find $\bar{\gamma}_{2}\geq 0$ such that
\begin{align*}
     \left \lvert \mathcal{F}^{k}(x_{k-1},u_{k}) \right \rvert \leq \bar{\gamma}_{2}.
\end{align*}
Assume that $\bar{\gamma} = \max \{\bar{\gamma}_{1},\bar{\gamma}_{2}\}$. Then,
\begin{align}
    & \tfrac{1}{2}\gamma^{2}\leq    -\mathcal{F}^{k}(x_k,u_{k})\leq  \bar{\gamma}\label{the1.53}\\
\noalign{\noindent \text{and}}
    & \left \lvert \mathcal{F}^{k}(x_{k-1},u_{k}) \right \rvert \leq \bar{\gamma}.\label{the1.54}
\end{align}
On the other hand, suppose $\mathcal{F}^{k}(x_k,u_{k}) \leq\mathcal{F}^{k}(x_{k-1},u_{k}) $. Then, we have 
\begin{align}\label{the1.55}
 \left \lvert\mathcal{F}^{k}(x_{k-1},u_{k}) - \mathcal{F}^{k}(x_k,u_{k}) \right \rvert & =\left \lvert \psi_{e}\left(\nabla f^{a_{k,\bar{j}_{k}}}(x_{k-1})^{\top}u_{k}\right)- \mathcal{F}^{k}(x_k,u_{k})\right\rvert, \nonumber\\ 
 &~~~ \text{ where }\bar{j}_{k}=\underset{j\in [\omega_{k}]}{\text{argmax}}\left\{\psi_{e}\left(\nabla f^{a_{k,j}}(x_{k-1})^{\top}u_{k}\right)\right\}\nonumber\\
& \leq \left\lvert\psi_{e}\left(\nabla f^{a_{k,\bar{j}_{k}}}(x_{k-1})^{\top}u_{k}\right)- \psi_{e}\left(\nabla f^{a_{k,\bar{j}_{k}}}(x_{k})^{\top}u_{k}\right)\right\rvert \nonumber\\
& \leq \bar{L} \left\|\nabla f^{a_{k,\bar{j}_{k}}}(x_{k-1})- \nabla f^{a_{k,\bar{j}_{k}}}(x_{k})\right\|\|u_{k}\| \nonumber\\
 & \leq \bar{L}L\tilde{\gamma}\|x_{k}-x_{k-1}\| \leq \bar{L}L\tilde{\gamma}\lambda,
 \end{align}
where  $L$ is a Lipschitz constant of $\nabla f^{i}$ for all  $i\in [p]$ and $\|s_{k-1}\|\leq \lambda$.
In a similar manner, if $\mathcal{F}^{k}(x_{k-1},u_{k}) \leq\mathcal{F}^{k}(x_{k},u_{k}) $, then
\begin{align}\label{the1.56}
 \left \lvert\mathcal{F}^{k}(x_{k},u_{k}) - \mathcal{F}^{k}(x_{k-1},u_{k}) \right \rvert &= \left\lvert \psi_{e}\left(\nabla f^{a_{k,\tilde{j}_{k}}}(x_{k})^{\top}u_{k}\right)- \mathcal{F}^{k}(x_{k-1},u_{k})\right\rvert, \nonumber\\ 
 &~~~ \text{ where }\tilde{j}_{k}=\underset{j\in [\omega_{k}]}{\text{argmax}}\left\{\psi_{e}\left(\nabla f^{a_{k,j}}(x_{k})^{\top}u_{k}\right)\right\} \nonumber\\
& \leq \left \lvert\psi_{e}\left (\nabla f^{a_{k,\tilde{j}_{k}}}(x_{k-1})^{\top}u_{k}\right)- \psi_{e}\left(\nabla f^{a_{k,\tilde{j}_{k}}}(x_{k})^{\top}u_{k}\right)\right\rvert 
\nonumber\\
 & \leq \bar{L}L\tilde{\gamma}\lambda.
\end{align}
Because of the relations \eqref{the1.55} and \eqref{the1.56}, we have
\begin{equation}\label{the1.57}
    \left \lvert\mathcal{F}^{k}(x_{k-1},u_{k}) - \mathcal{F}^{k}(x_k,u_{k}) \right \rvert \leq \bar{L}L\tilde{\gamma}\lambda.
\end{equation}
\noindent Define $b:=\tfrac{4\bar{\gamma}}{\gamma^{2}}$ and $\lambda:=\tfrac{\gamma^{2}}{2\bar{L}L\tilde{\gamma}b}$ for the $\mathrm{PRP}$ method. By \eqref{the1.53} and \eqref{the1.54}, we get

\begin{equation*}
    \lvert\beta_{\mathrm{PRP}_{k}}\rvert \leq \tfrac{-\mathcal{F}^{k}(x_k,u_{k})+\lvert\mathcal{F}^{k}(x_{k-1},u_{k})\rvert}{-\mathcal{F}^{k-1}(x_{k-1},u_{k-1})}\leq \tfrac{4\bar{\gamma}}{\gamma^{2}}=b.
\end{equation*}
Now, if $\|s_{k-1}\|\leq \lambda$, then from \eqref{the1.53} and \eqref{the1.57}, it follows

\begin{equation*}
    \lvert\beta_{\mathrm{PRP}_{k}}\rvert \leq \tfrac{L\bar{L}\tilde {\gamma} \lambda}{\gamma^{2}}=\tfrac{1}{2b}.
\end{equation*}
Thus, $\mathrm{PRP}$ method has Property $(\ast)$.\\

Let us now examine the $\mathrm{HS}$ method. By the standard Wolfe condition \eqref{stan, wol.} and the sufficient descent condition \eqref{suf_d}, we obtain
\begin{align}
  \mathcal{F}^{k-1}(x_{k},d_{k-1}) -  \mathcal{F}^{k-1}(x_{k-1},d_{k-1}) &\geq \sigma \mathcal{F}^{k-1}(x_{k-1},d_{k-1})-\mathcal{F}^{k-1}(x_{k-1},d_{k-1}) \nonumber \\
  & =(\sigma-1)\mathcal{F}^{k-1}(x_{k-1},d_{k-1}) \nonumber\\
  & \geq c(\sigma-1)\mathcal{F}_{k-1}(x_{k-1},u_{k-1}) \label{the1.58}\\
  & \geq \tfrac{c(1-\sigma)\gamma^{2}}{2}> 0. \nonumber
\end{align}
For the HS method, define $b:=\tfrac{4\bar{\gamma}}{c(1-\sigma)\gamma^{2}}$ and $\lambda:=\tfrac{c(1-\sigma)\gamma^{2}}{4\bar{L}L\tilde{\gamma}b}$.\\ \\ 
 From relations \eqref{the1.53}, \eqref{the1.54}, and \eqref{the1.58}, we have
\begin{equation*}
     \lvert\beta_{\mathrm{HS}_{k}}\rvert \leq \tfrac{-\mathcal{F}^{k}(x_k,u_{k})+\lvert\mathcal{F}^{k}(x_{k-1},u_{k})\rvert}{\mathcal{F}^{k-1}(x_{k},d_{k-1}) -  \mathcal{F}^{k-1}(x_{k-1},d_{k-1})}\leq \tfrac{4\bar{\gamma}}{c(1-\sigma)\gamma^{2}}=b,
\end{equation*}
and when $\|s_{k-1}\|\leq \lambda$, it follows from \eqref{the1.57} and \eqref{the1.58} that 
\begin{equation*}
    \lvert\beta_{\mathrm{HS}_{k}}\rvert \leq \tfrac{L\bar{L}\tilde {\gamma} \lambda}{\gamma^{2}}=\tfrac{1}{2b},
\end{equation*}
which concludes that the $\mathrm{HS}$ method has Property $(\ast)$.    
\end{proof}

\section{Numerical experiments}\label{Numerical Experiments}
This section presents the numerical performance of the proposed Algorithm \ref{algo} when the $\beta_k$-value is chosen by the proposed DY, PRP and HS rules. All the computational experiments are carried out through MATLAB R2023b software installed on a laptop with Windows 11 operating system equipped with a 3.20 GHz Intel Core i5 CPU and 8 GB memory. The details of the experimental setup and parameters used during the implementation are given below. \\

\begin{itemize}
\item Out of the five test instances shown below, in Examples \ref{EX_1}, \ref{Ex_2} and \ref{Ex_3}, the ordering cone $K$ is taken as the usual nonnegative hyper-octant $\mathbb{R}^{m+}$. In Examples \ref{Ex_4} and \ref{Ex_5}, we consider problems with general cones that are not $\mathbb{R}^{m+}$. Moreover, one of the cones considered in Example \ref{Ex_5} is nonfinitely generated. \\

\item In all of the considered examples, except in one problem in Example \ref{Ex_5}, the parameter $ e \in \mathrm{int}(K)$ of the function $\psi_{e}$ is taken as $e = (1, 1, \ldots, 1)^{\top}$. In Example \ref{Ex_5}, for the problem with a nonfinitely generated cone, we take $e = (0, 0, 1)^{\top}$ since, in this case, $e = (1, 1, 1)^{\top}$ is not an interior point of $K$. \\

\item In Step 1 of Algorithm \ref{algo}, to compute the set $M_k:=\mathrm{Min}(F(x_k),K)$, we use the crude way of pair-wise comparing the elements in $F(x_k)$. \\

\item We use the stopping criterion in Step 3 of Algorithm \ref{algo} as $\|u_{k}\|<10^{-4}$. \\

\item For implementing the line search given in Step 5 of Algorithm \ref{algo}, we follow the strategy described in \cite{lucambio2019wolfe} with the parameters as $\alpha_{0}=1,~ \alpha_{\mathrm{max}}=100,~\rho=10^{-4}$, and $\sigma = 0.1$. \\

\item To evaluate the performance of the methods, we run the MATLAB code of Algorithm \ref{algo} for arbitrarily chosen $100$ initial points and calculate the min (minimum), mean, and max (maximum) of the following two metrics: 

\begin{itemize}
    \item \emph{Time}: The time (in seconds) taken by the algorithm to reach the stopping condition for each initial point.

    \item \emph{Iteration counts}: The number of iterations taken by the algorithm to reach the stopping condition for each initial point. \\ 
\end{itemize}

\item We compare the performance of the proposed DY, PRP, and HS methods with the FR (Fletcher-Reeves) and CD (Conjugate Descent) methods given in \cite{kumar2024nonlinear}. This comparison is exhibited for all the examples except for Example \ref{Ex_5} with a nonfinitely generated cone $K$ because the methods in \cite{kumar2024nonlinear} were derived only for finitely generated cones $K$. Note that we do not compare the proposed methods with the existing steepest descent method \cite{bouza2021steepest} for set optimization since Kumar et al.  \cite{kumar2024nonlinear} reported that FR and CD outperform the steepest descent method. \\

%\item For each of the five set-optimization problem examined below, except the Experiment \ref{EX_1}, the initial and terminal iterates of the seqeuence $\{x_k\}$ generated by a method are presented by red color and green color, respectively. We use the shapes $\bullet,~\boldsymbol{\star},$ and $\blacktriangle$ to represent the values of $F$ at different initial points. The intermediate iterates for these initial points are shown using cyan, magenta, and green colors. The initial points are indicated in black, while the termination points are in red. For instance, if an initial point is marked by a black bullet $\bullet$, the corresponding termination point is depicted by a red bullet {\red $\bullet$}, and the intermediate iterates are shown using cyan {\cyan $\bullet$}, magenta {\magenta $\bullet$}, or green {\green $\bullet$} bullets. This means that the same shape is used to depict the entire sequence of iterates generated by Algorithm \ref{algo1}.

\item In each example, we depict the sequence of sets $\{F(x_k)\}$, where $\{x_k\}$ is the sequence of iterates generated by the method that outperforms the considered set optimization problem. In depicting the movement of the sequence $\{x_k\}$ or $\{F(x_k)\}$, we use red color to indicate the initial point, blue for the intermediate point, and green for the terminal point. \\ 

\end{itemize}

The first problem in the experiment is taken from \cite{bouza2021steepest}, which comes from the robust counterpart of a vector-valued facility location problem under uncertainty.\\

\begin{example}\label{EX_1}
     \emph{\cite[Test Instance 5.2]{bouza2021steepest} Consider the function $F: \mathbb{R}^{2} \rightrightarrows \mathbb{R}^{3}$ defined by}
     \[F(x):=\left\{f^{1}(x), f^{2}(x), \ldots,f^{100}(x)\right\},\]
\end{example}
\noindent
where the expression of $f^{i}(x), i=1,2,\ldots,100$, is given by

\[f^{i}(x):=\tfrac{1}{2}\left(\begin{aligned}
    & \|x-l_{1}-u_{i}\|^{2}\\
    & \|x-l_{2}-u_{i}\|^{2}\\
    & \|x-l_{3}-u_{i}\|^{2}
\end{aligned}\right),\]
where 
$l_1 :=\begin{pmatrix} 0 \\ 0 \end{pmatrix}$, 
$l_2 :=\begin{pmatrix} 0 \\ 8 \end{pmatrix}$, 
$l_3 :=\begin{pmatrix} 8 \\ 0 \end{pmatrix}$
and $\{u_{1},u_{2},\ldots,u_{100}\}$ is an enumeration of the set $\mathcal{U}:=\mathcal{U}_{1} \times \mathcal{U}_{1}$ with
\[\mathcal{U}_{1} := \left\{-1,-0.7778,-0.5556,-0.3333,-0.1111,0.1111,0.3333, 0.5556,0.7778,1 \right\}.\]
%\noindent \textbf{Experiment 1}
%
%
%
Note that the set of the local weakly minimal points of $F$ is 
\[\mathrm{conv}\left\{(l_{1}+\mathcal{U})\cup (l_{2}+\mathcal{U}) \cup (l_{3}+\mathcal{U})\right\}.\]
 
For the evaluation of the performance of an algorithm, we have arbitrarily picked 100 initial points from the set $[-50,~ 50] \times [-50, ~50]$. Then, we run all five conjugate gradient methods (proposed three---DY, PRP and HS---and two from \cite{kumar2024nonlinear}) for the same 100 initial points.  
A summary of the performance of all the five methods is provided in Table \ref{table_ex1}. In this example, we observe that the performance of all the methods is almost identical. \\ 

The set of solutions generated by the HS method for the chosen 100 initial points is shown in Fig. \ref{fig_ex1} by green-colored bullets. In Fig. \ref{fig_ex1}, three bigger size black-colored bullets represent the locations $l_{1},~l_{2}, \text{ and } l_{3}$; the three bunch of gray bullets collectively represent the set  
$(l_{1}+\mathcal{U})\cup (l_{2}+\mathcal{U}) \cup (l_{3}+\mathcal{U})$. Notice that all the generated green-colored bullets are weakly minimal points of the set-valued function $F$. \\

\begin{table}[!h]
\centering
\begin{tabular}{p{1.1cm} p{1.1cm} p{1cm} p{1.1cm} p{1.1cm} p{1.1cm} p{1.1cm}}
\hline

\hline 
\multicolumn{1}{c}{\text{Method}}& \multicolumn{3}{c}{\text{Iteration counts}} & \multicolumn{3}{c}{\text{Time}} \\
\cmidrule(rl){2-4} \cmidrule(rl){5-7}
& min & mean & max & min & mean & max\\
\hline

\hline
DY & 1 & 1.04 & 3  & 0.3457 & 1.6164 & 14.5765  \\
PRP & 1 & 1.03 & 2  & 0.3457 & 1.4872 & 8.7392\\
HS & 1 & 1.03 & 2  & 0.3457 & 1.4238 & 10.7682\\
FR & 1 & 1.04 & 3  & 0.3457 & 1.5256 & 10.6281\\
CD & 1 & 1.04 & 3  & 0.3457 & 2.0030 & 22.2712\\

\hline 

\hline
\end{tabular}
    \caption{Performance of conjugate gradient methods on Example  \ref{EX_1}}\label{table_ex1}
\end{table}

\begin{figure}[h!]
\captionsetup[subfigure]{labelformat=empty}
\centering
\subfloat[]{
\includegraphics[width=0.6\textwidth]{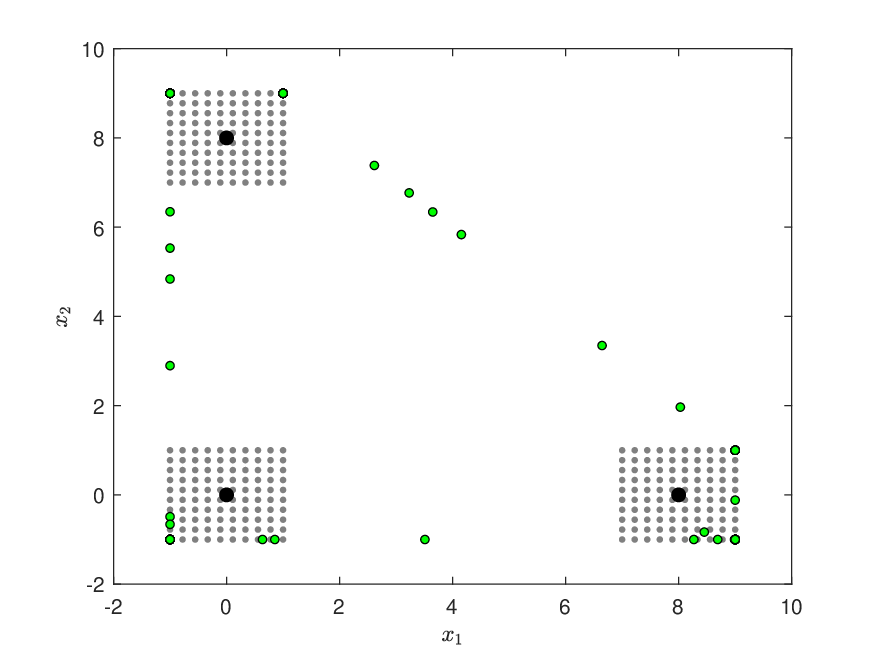}
}
\vspace{-0.8cm}
\caption{The solution set (green points) in the decision space of Example \ref{EX_1} generated by Algorithm \ref{algo} with HS rule}\label{fig_ex1}
\end{figure}

\begin{example}\label{Ex_2}
\emph{In this problem, we consider the function $F: \mathbb{R}^{2} \rightrightarrows \mathbb{R}^{2}$ defined by}
\[F(x):=\Bigl\{f^{1}(x), f^{2}(x), \ldots,f^{100}(x)\Bigr\},\]
\end{example}
\noindent
where the expression of $f^{i}(x), i=1,2,\ldots,100$, is given by
\[f^{i}(x_1, x_2) :=\mqty(
    \sin(x_{1})+x^{2}_{1}(1+\cos(x_{2}))+ 2x_{1} \cos(x_{2})\cos\left(\tfrac{\pi (i-1)}{25}\right)\sin^{2}\left(\tfrac{ \pi(i-1)}{100}\right)\\
    \cos(x_{2})+x^{2}_{2}(2+\cos(x_{1})) + x_{1}\sin(x_{2})\sin\left(\tfrac{\pi (i-1)}{25}\right)\cos^{2}\left(\tfrac{ \pi(i-1)}{100}\right)
).\]

In this example, computer codes of all the five conjugate gradient methods are run for the same set of 100 initial points, which are arbitrarily chosen from the domain $[-\pi,~ \pi] \times [-\pi, ~\pi]$. A summary of the performance of the five methods is presented in Table \ref{Table_ex2}. We note that for this example, Algorithm \ref{algo} with PRP rule performs better than the other four methods.   \\

\begin{table}[!h]
\centering
\begin{tabular}{p{1.1cm} p{1.1cm} p{1cm} p{1.1cm} p{1.1cm} p{1.1cm} p{1.1cm}}
\hline

\hline 
\multicolumn{1}{c}{\text{Method}}& \multicolumn{3}{c}{\text{Iteration counts}} & \multicolumn{3}{c}{\text{Time}} \\
\cmidrule(rl){2-4} \cmidrule(rl){5-7}
\vspace{0.1cm}
& min & mean & max & min & mean & max\\
\hline

\hline

DY & 0 & 11.02 & 84  & 1.3270 & 50.4930 & 367.7961  \\
PRP & 0 & 5.52 & 40  & 1.3270 & 24.6789 & 157.8450\\
HS & 0 & 6.78 & 91  & 1.3270 & 36.9255 & 331.9024\\
FR & 0 & 11.98 & 195  & 1.3270 & 55.3178 & 1210.6040\\
CD & 0 & 7.94 & 62  & 1.3270 & 82.8473 & 335.9404\\

\hline 

\hline
\end{tabular}
    \caption{Performance of conjugate gradient methods on Example  \ref{Ex_2}}\label{Table_ex2}
\end{table}

In Figure \ref{Fig_ex2_D}, we exhibit the movement of the sequence of iterates $\{x_k\}$  generated by Algorithm \ref{algo} with PRP rule corresponding to the following four initial points ($x_0$): 
\begin{equation}\label{4_initial point}
    (-1.1812,   -2.5496)^{\top},~ (2.5971,    -2.1842)^{\top},~ (2.4561,  2.8858)^{\top} \text{ and }(2.7947,    -0.0574)^{\top}. 
\end{equation}
In all the four subfigures in Fig. \ref{Fig_ex2_D}, the initial point of the sequence $\{x_k\}$ is depicted by red color, the terminal by green, and the intermediate ones by blue. \\

\begin{figure}[h!]
\captionsetup[subfigure]{labelformat=empty}
\centering
\subfloat[]{
\includegraphics[width=0.8\textwidth]{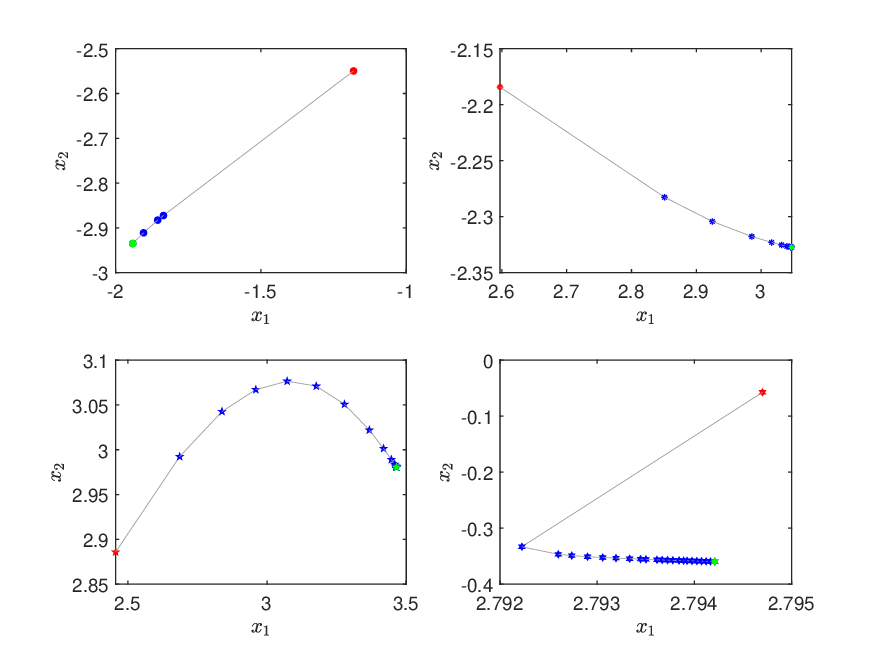}
}
\caption{Movement of the sequence $\{x_k\}$ generated by the PRP method for  Example \ref{Ex_2} corresponding to the four initial points given in \eqref{4_initial point}}\label{Fig_ex2_D}
\end{figure}

The movement of the sequence $\{F(x_k)\}$ (in the image space) corresponding to the four different sequences plotted in Fig. \ref{Fig_ex2_D} are depicted in the four different subfigures in Fig. \ref{Fig_ex2_o}. The top-left, top-right, bottom-left, and bottom-right subfigure in Fig. \ref{Fig_ex2_o} are corresponding to the top-left, top-right, bottom-left, and bottom-right subfigure in Fig. \ref{Fig_ex2_D}, respectively. \\

The red-colored bullet point in the top-left subfigure of Fig. \ref{Fig_ex2_D} is the initial point $x_0 = (-1.1812,   -2.5496)^{\top}$; the bunch of red-colored bullets in the top-left subfigure of Fig. \ref{Fig_ex2_o} is the set $F(x_0)$. The green-colored bullet in the top-left subfigure of Fig. \ref{Fig_ex2_D} is the terminal point ($\tilde{x}$ say) of the sequence $\{x_k\}$; the bunch of green-colored bullets in the top-left subfigure of Fig. \ref{Fig_ex2_o} is the set $F(\tilde{x})$. Similarly, the set $F(x_k)$ corresponding to the intermediate points in $\{x_k\}$ are also depicted. \\ 

In the same fashion as the correspondence of the top-left subfigures of Fig. \ref{Fig_ex2_D} and Fig. \ref{Fig_ex2_o}, the top-right, bottom-left and bottom-right subfigures in Fig. \ref{Fig_ex2_o} are generated. \\

\begin{figure}[h!]
\captionsetup[subfigure]{labelformat=empty}
\centering
\subfloat[]{
\includegraphics[width=0.8\textwidth]{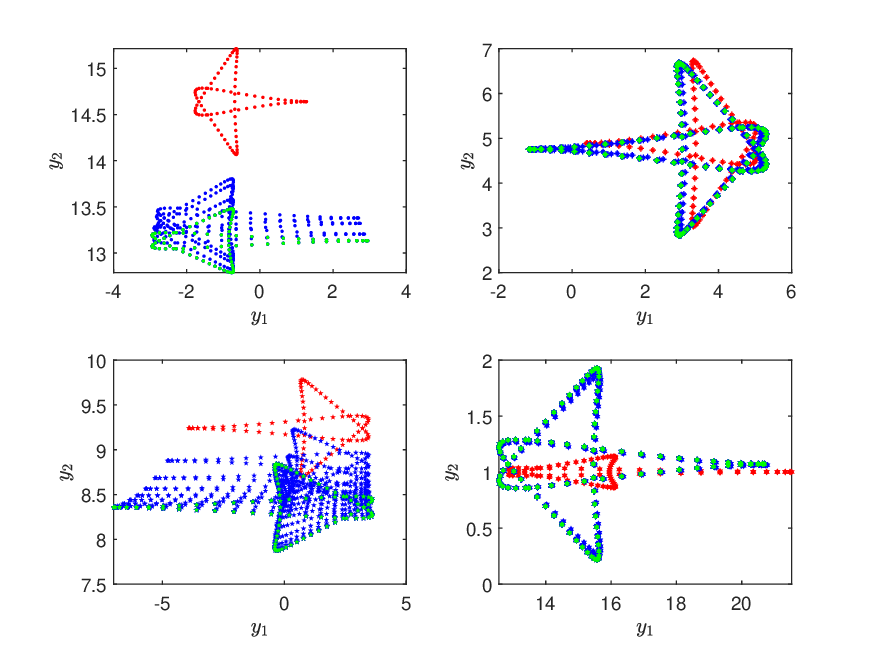}}
\vspace{-0.8cm}
\caption{The sequence $\{F(x_k)\}$ on the image space of $F$ in Example  \ref{Ex_2} corresponding to the sequence $\{x_k\}$ in Fig. \ref{Fig_ex2_D}}\label{Fig_ex2_o}
\end{figure}

\begin{example}\label{Ex_3}
\emph{We consider the function $F: \mathbb{R}^{3} \rightrightarrows \mathbb{R}^{3}$ defined as}
\end{example}
\[F(x):=\Bigl\{f^{1}(x), f^{2}(x), \ldots,f^{100}(x)\Bigr\},\]
\noindent 
where $f^{i}(x):=g(x)+h^{i}(x), i=1,2,\ldots,100$, 
\begin{align*}
   g(x_1, x_2) := \mqty(
    \tfrac{1}{2}(x_{1}-2)^{4} + \tfrac{1}{13} (x_{2}+ 1)^{2}+3\\ \tfrac{1}{36}(x_{1} + x_{2} -3)^2 + \tfrac{1}{18}(-x_{1} + x_{2} + 2)^2 -17\\\tfrac{1}{175}(x_{1} + 2 x_{2} -1)^2 + \tfrac{1}{17}(-x_{1} + 2 x_{2})^2) 
\end{align*}
     and
     \begin{align*}
        h^{i}(x_1, x_2) := \tfrac{1}{100}\mqty(
    \exp\left(\tfrac{x_{1}}{2}\right)\cos(x_{2})+ x_{1}\cos(x_{2})\sin^{3}\left(\tfrac{2 \pi(i-1)}{100}\right)-x_{2}\sin(x_{2})\cos \left(\tfrac{2 \pi(i-1)}{100}\right)\\
    \exp\left(\tfrac{x_{2}}{100}\right)\sin(x_{1})+ x_{1}\sin(x_{2})\sin^{3}\left(\tfrac{2 \pi(i-1)}{100}\right)+x_{2}\cos(x_{2})\cos\left(\tfrac{2 \pi(i-1)}{100}\right)\\
     \sin^{2}(x_{3})\sin^{3}\left(\tfrac{2\pi (i-1)}{100}\right)). 
     \end{align*}
Note that 
\[F(x)=g(x)+\Bigl\{h^{1}(x), h^{2}(x),\ldots,h^{100}(x)\Bigr\}.\]
Here, the function \(g\) is derived from the well-known MOP7 test problem \cite{huband2006review} for vector optimization. 
%The images of the set-valued mapping in this problem are discretized, shifted, and rotated, as illustrated in Fig. \ref{Fig_ex3}. 
The computer codes of all five conjugate gradient methods were executed using the same set of 100 initial points, arbitrarily selected from the domain $[-500,~ 500] \times [-500,~ 500]$. A performance summary of these five methods is provided in Table \ref{Table_ex3}. Notably, for this example, Algorithm \ref{algo} implementing HS rule outperformed the other four methods. \\ 

In Fig. \ref{Fig_ex3}, we exhibit the sequence of set-valued iterates $\{F(x_k)\}$ corresponding to the sequence $\{x_k\}$, which is generated by Algorithm \ref{algo} with HS rule for the initial point $x_0 = (1.3908,~ 0.5130,~   64.9686)^{\top}$. The bunch of red-colored bullets, in Fig. \ref{Fig_ex3}, is the set $F(x_0)$, the blues ones are intermediate $F(x_k)$'s, and the green bunch represents the terminal $F(x_k)$. \\

\begin{table}[!h]
\centering
\begin{tabular}{p{1.1cm} p{1.1cm} p{1cm} p{1.1cm} p{1.1cm} p{1.1cm} p{1.1cm}}
\hline

\hline 
\multicolumn{1}{c}{\text{Method}}& \multicolumn{3}{c}{\text{Iteration counts}} & \multicolumn{3}{c}{\text{Time}} \\
\cmidrule(rl){2-4} \cmidrule(rl){5-7}
\vspace{0.1cm}
& min & mean & max & min & mean & max\\
\hline

\hline

DY & 3 & 42.5938 & 142  & 64.7456 & 218.7738 & 1012.4873  \\
PRP & 5 & 23.0729 & 266  & 51.1373 & 157.0143 & 1747.0088\\
HS & 3 & 17.5833 & 204  & 44.2141 & 138.9676 & 1270.7534\\
FR & 4 & 36.3500 & 194  & 54.8662 & 164.8179 & 828.6454\\
CD & 3 & 41.2641 & 430  & 67.6458 & 174.9024 & 1811.7712\\
\hline 

\hline

\end{tabular}
 \caption{Performance of conjugate gradient methods on Example  \ref{Ex_3}}\label{Table_ex3}
\end{table}

\vspace{-1.5cm}

\begin{figure}[h!]
\captionsetup[subfigure]{labelformat=empty}
\centering
\subfloat[]{
\includegraphics[width=0.6\textwidth]{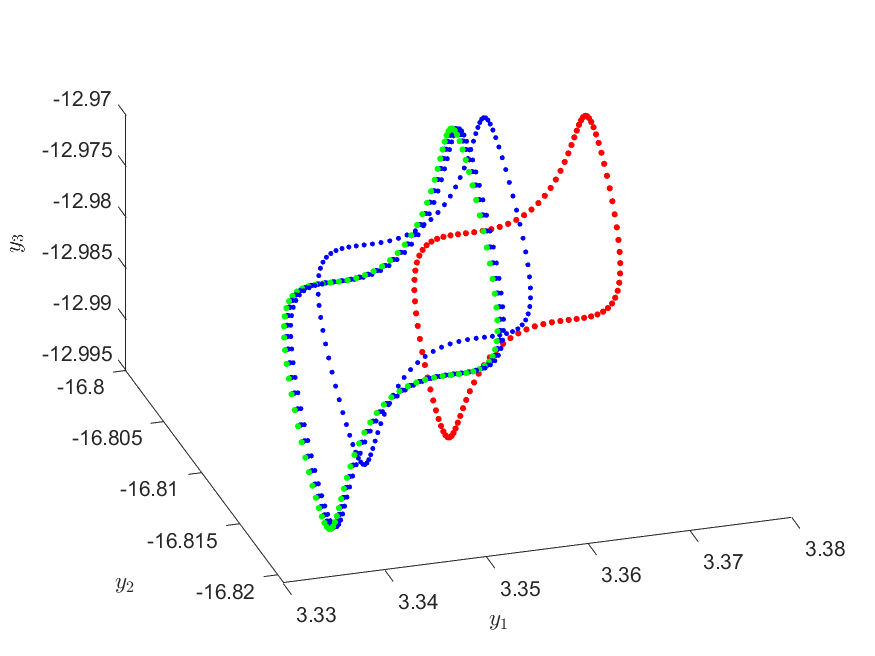}
}
\caption{Sequence generated by Algorithm \ref{algo} with HS rule, in the image space of Expample \ref{Ex_3} for the initial point $x_0 = (1.3908,~ 0.5130,~   64.9686)^{\top}$ } \label{Fig_ex3}
\end{figure}

In the next experiment, the set-valued mapping is taken from \cite[Test Instance 5.1]{bouza2021steepest}.\\

\begin{example}\label{Ex_4}
    \emph{
    Consider $F: \mathbb{R} \rightrightarrows \mathbb{R}^{2}$, defined as }
\begin{equation}\label{Ob_ex_4}
  F(x) := \left\{f^{1}(x), f^{2}(x), \ldots,f^{5}(x)\right\},  
\end{equation}
\end{example}
\noindent
where for each $i\in [5],~f^{i}:\mathbb{R}\rightarrow \mathbb{R}^2 $ is given by 
\[f^{i}(x) := \mqty(
    x \\
     \tfrac{x}{2}\sin(x))+\sin^{2}(x)\left[\tfrac{(i-1)}{4}  \mqty(
     1 \\
     -1
)+\left(1-\tfrac{(i-1)}{4}\right)\mqty(
    1\\
     -1
)\right]. \]
In this example, we consider two different set optimization problems \eqref{SPL} with the same objective function $F$ but with different ordering cones $K_1$ and $K_2$, where  
\begin{subequations}\label{cone_ex4}
\begin{align} 
    & K_{1} := \left\{(y_{1},y_{2})^{\top}~\middle|~y_{1}\geq 0,~ y_{2}\geq 0\right\}\label{cone_Ex_4_k1}\\
    \noalign{\noindent \text{or}}
     & K_{2} := \left\{(y_{1},y_{2})^{\top}~\middle|~-y_{1}+3y_{2}\geq 0,~ 3y_{1}-y_{2}\geq 0\right\}.\label{cone_Ex_4_k2}    
\end{align}  
\end{subequations}
\noindent 
%In this case, the objective values are discretized segments that move along a curve and are dilated by a factor dependent on the argument. Since we are investigating this problem with two different cones $K_{1}$ and $K_{2}$ given in \eqref{cone_ex4}, accordingly 
%Thus, we have two \eqref{SPL} problems in this problem. One is with cone $K_{1}$, and the other is with cone $K_{2}$. However, the objective function of the problems is the same set-valued mapping $F$ given by \eqref{Ob_ex_4}. \\ 
For both problems, we run computer codes of all the five conjugate gradient methods for the same 100 initial points chosen arbitrarily from the interval \([-5\pi, 5\pi]\). A summary of the performance of the five methods is given in Table \ref{Table_ex4_k1} and Table \ref{Table_ex4_k2} corresponding to cones $K_{1}$ and $K_{2}$, respectively. As seen from Tables \ref{Table_ex4_k1} and \ref{Table_ex4_k2}, although the performance of all five methods is almost identical, the HS rule slightly outperforms the other four methods. \\

The sequence of sets $\{F(x_k)\}$ corresponding to the generated sequence $\{x_k\}$ by Algorithm \ref{algo} with the HS rule and the initial point $x_{0}=-10.4$ is depicted in Fig. \ref{Fig_ex4}. The five curves in both the subfigures of Fig. \ref{Fig_ex4} are graphs of the functions $f^i$, $i = 1, 2, \ldots, 5$. The collection of the red-colored bullet points is the set $F(x_0)$, and the collection of green-colored bullet points is the value of $F$ at the terminal iterate. \\

The collection of the green bullet points on the right subfigure of Fig. \ref{Fig_ex4} is the set $F(x_0)$, where $x_{0}=-10.4$. From this subfigure, we see that the green curve and the violet curve do not have a local portion in the sky-blue-shaded region. Thus, $x_{0}=-10.4$ is a local weakly minimal point of the considered set optimization problem with the ordering cone $K_2$. Hence, Algorithm \ref{algo} gets terminated at the initial point $x_0$.   \\

From the left subfigure of Fig. \ref{Fig_ex4}, we see that the bunch of red-colored bullet points is the set $F(x_0)$, which is not a local weakly minimal solution of the considered set optimization problem with the ordering cone $K_1$ since all the five graphs of the functions $f^1$, $f^2, \ldots, f^5$ have local portion inside the set $F(x_0) - K_1$. Thus, $x_{0}=-10.4$ is not a local weakly minimal point of the considered set optimization problem with the ordering cone $K_1$. Hence, for $K_1$, Algorithm \ref{algo} does not get terminated at the initial point $x_0$. In fact, for $K_1$, Algorithm \ref{algo} with HS rule gets terminated in the next iterate at which the objective function value is the set of all green-colored bullet points in the left subfigure of Fig. \ref{Fig_ex4}. \\

\begin{table}[!h]
\centering
\begin{tabular}{p{1.1cm} p{1.1cm} p{1cm} p{1.1cm} p{1.1cm} p{1.1cm} p{1.1cm}}
\hline

\hline 
\multicolumn{1}{c}{\text{Method}}& \multicolumn{3}{c}{\text{Iteration counts}} & \multicolumn{3}{c}{\text{Time}} \\
\cmidrule(rl){2-4} \cmidrule(rl){5-7}
\vspace{0.1cm}
& min & mean & max & min & mean & max\\
\hline

\hline
DY & 0 & 1.01 & 4  & 0.2614 & 2.6601 & 13.1601 \\
PRP & 0 & 1.04 & 4  & 0.2614 & 2.8637 & 13.7649\\
HS & 0 & 0.99 & 4  & 0.2614 & 2.5884 & 12.8843\\
FR & 0 & 1.02 & 5  & 0.2614 & 2.8565 & 17.1513\\
CD & 0 & 1.01 & 5  & 0.2614 & 2.7202 & 16.3896\\

\hline 

\hline
\end{tabular}
    \caption{Performance of conjugate gradient methods on Example \ref{Ex_4} associated with the cone $K_{1}$ given in \eqref{cone_Ex_4_k1}}\label{Table_ex4_k1}
\end{table}

%\vspace{-0.6cm}

\begin{table}[!h]
\centering
\begin{tabular}{p{1.1cm} p{1.1cm} p{1cm} p{1.1cm} p{1.1cm} p{1.1cm} p{1.1cm}}
\hline

\hline 
\multicolumn{1}{c}{\text{Method}}& \multicolumn{3}{c}{\text{Iteration counts}} & \multicolumn{3}{c}{\text{Time}} \\
\cmidrule(rl){2-4} \cmidrule(rl){5-7}
\vspace{0.1cm}
& min & mean & max & min & mean & max\\
\hline

\hline

DY & 0 & 0.04 & 1  & 0.2863 & 0.3188 & 5.0179  \\
PRP & 0 & 0.04 & 1  & 0.2863 & 0.3188 & 5.0179\\
HS & 0 & 0.04 & 1  & 0.2863 & 0.3188 & 5.0179\\
FR & 0 & 0.04 & 1  & 0.2863 & 0.3188 & 5.0179\\
CD & 0 & 0.04 & 1  & 0.2863 & 0.3188 & 5.0179\\

\hline 

\hline
\end{tabular}
    \caption{Performance of conjugate gradient methods on Example \ref{Ex_4} associated with the cone $K_{2}$ given in \eqref{cone_Ex_4_k2}}\label{Table_ex4_k2}
\end{table}

\begin{figure}[h!]
\captionsetup[subfigure]{labelformat=empty}
\centering
\subfloat[]{
\includegraphics[width=0.49\textwidth]{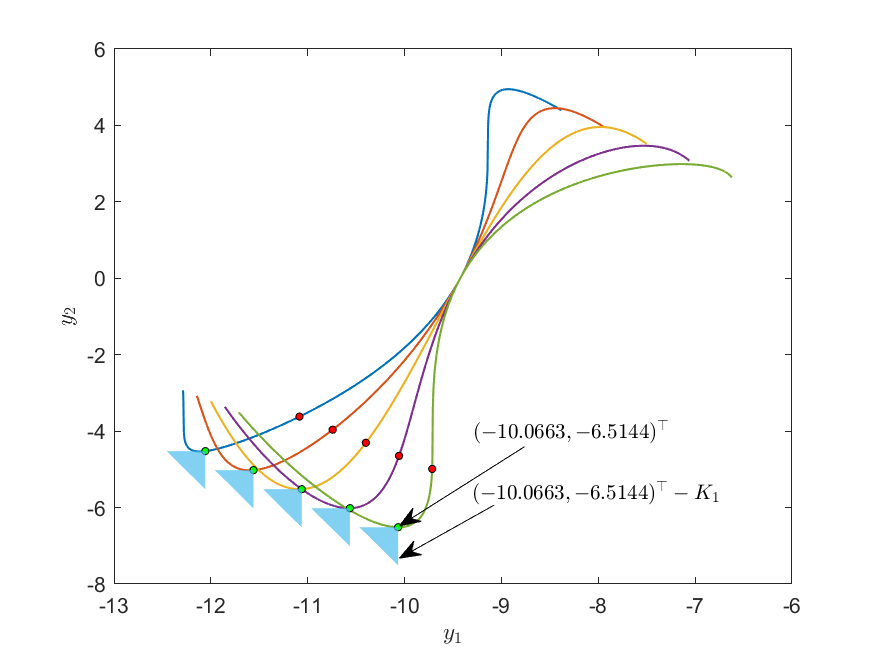}
\includegraphics[width=0.49\textwidth]{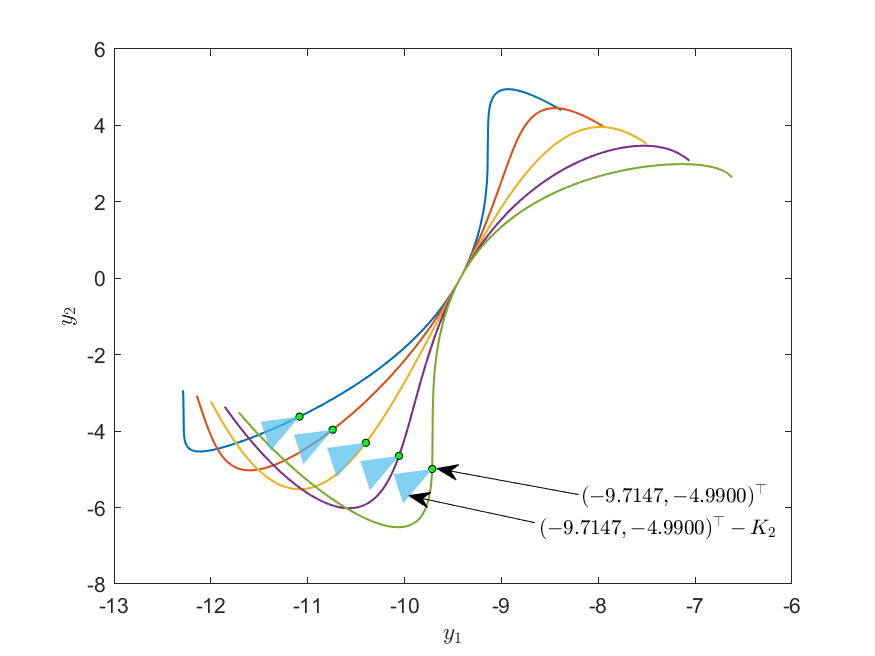}
}
\vspace{-0.6cm}
\caption{Sequence generated by Algorithm \ref{algo} with HS rule, in the image space of Example \ref{Ex_4}, for the initial point $x_0 = -10.4$ corresponding to the cones $K_{1}$ and $K_{2}$ as given in \eqref{cone_ex4}}\label{Fig_ex4}
\end{figure}

Next, we consider an example where the image space of the objective function is $\mathbb{R}^3$, but the ordering cone $K$ is not the usual $\mathbb{R}^{3+}$. \\

\begin{example}\label{Ex_5}
\emph{Similar to Example \ref{Ex_4}, consider two set-valued optimization problems with the identical objective function $F: \mathbb{R} \rightrightarrows \mathbb{R}^{3}$ but with two different cones:}    
\begin{subequations}\label{Cone_ex5}
    \begin{align}
    & K_{1} : = \left\{(y_{1},y_{2},y_{3})^{\top} ~\middle|~ y_{1}\geq 0,~ y_{2}\geq 0, y_{3}\geq 0\right\} \label{Cone_ex5_k1} \\
    \noalign{\noindent \text{\emph{and}}}
     & K_{2} : = \left\{(y_{1},y_{2},y_{3})^{\top} ~\middle|~ y_{3}\geq \sqrt{y^{2}_{1}+y^{2}_{2}} \right\},  \label{Cone_ex5_k2}
\end{align}
\end{subequations}   
\end{example}
\noindent
where 
\begin{equation}\label{Obj_ex5}
F(x):=\Bigl\{f^{1}(x), f^{2}(x), \ldots,f^{5}(x)\Bigr\},   
\end{equation}
\noindent 
with  
\[f^{i}(x) := \mqty(
     \tfrac{x}{2}\sin(x) + \cos^{2}(x)\tfrac{(i-3)}{2}\\
     \tfrac{1}{2}\cos(2x) + \sin^{2}(x)\tfrac{(-i+3)}{4}\\
     x\sin(2x) + \sin^{2}(x)\tfrac{(-i+3)}{2} \\
), i = 1, 2, \ldots, 5. \]
%
%
%
% As we explore this problem with two cones $K_{1}$ and $K_{2}$ given in \eqref{Cone_ex5}, accordingly there are two set optimization problems: one with cone $K_{1}$ and the other with cone $K_{2}$, both associated with the same set-valued mapping $F$ given in \eqref{Obj_ex5}.

\smallskip 
\smallskip
\noindent
Note that the cone $K_1$ is finitely generated, but $K_2$ is not finitely generated.  \\

For the problem with $K_1$, we run the computer code of all the five conjugate gradient methods for the same 100 initial points that are arbitrarily picked from the interval \([-15.5, -8]\). A summary of the performance of the five methods is presented in Table \ref{Table_ex5_k1}. Notice from Table \ref{Table_ex5_k1} that although all the methods perform equivalently, the DY rule has just a slightly better performance than the other four methods. \\

For the problem with $K_2$, we run the computer code of all the proposed three conjugate gradient methods for the same, arbitrarily picked, 100 initial points from \([-15.5, -8]\). A summary of the performance is provided in Table \ref{Table_ex5_k2}. Notice from Table \ref{Table_ex5_k2} that the DY rule has better performance than the other two methods. As the approaches described in \cite{kumar2024nonlinear} are not applicable for nonfinitely generated cones, for the problem with $K_2$, we do not compare the proposed three conjugate gradient rules with the FR and CD methods in \cite{kumar2024nonlinear}. \\

The sequence of sets $\{F(x_k)\}$ corresponding to the sequence $\{x_k\}$ generated by Algorithm \ref{algo} with DY rule for the initial point $x_{0}= -10.9$ are exhibited in Fig. \ref{fig_ex5}. The left subfigure corresponds to $K_1$, and the right corresponds to $K_2$. The curves in both the subfigures are the graphs of the functions $f^1$, $f^2, \ldots, f^5$. The bunch of red-colored bullet points is the set $F(x_0)$, and the green bunch is the value of the function $F$ at the terminal iterate. Blue-colored points are used to depict the value of $F$ at the intermediate iterates. Note that for the same initial point $x_{0}$, the solutions generated by the same method are different for different ordering cones $K_{1}$ and $K_{2}$. \\

\vspace{-0.6cm}
  
\begin{table}[!h]
\centering
\begin{tabular}{p{1.1cm} p{1.1cm} p{1cm} p{1.1cm} p{1.1cm} p{1.1cm} p{1.1cm}}
\hline

\hline 
\multicolumn{1}{c}{\text{Method}}& \multicolumn{3}{c}{\text{Iteration counts}} & \multicolumn{3}{c}{\text{Time}} \\
\cmidrule(rl){2-4} \cmidrule(rl){5-7}
\vspace{0.1cm}
& min & mean & max & min & mean & max\\
\hline

\hline

DY & 0 & 0.64 & 5  & 0.1790 & 1.9988  & 14.9486  \\
PRP & 0 & 0.66 & 5  & 0.1790 & 2.1551 & 15.3702\\
HS & 0 & 0.65 & 5  & 0.1790 & 2.0057 & 15.0421\\
FR & 0 & 0.65 & 5  & 0.1790 & 2.0149 & 17.7405\\
CD & 0 & 0.64 & 5  & 0.1790 & 2.0013 & 16.8554\\

\hline 

\hline
\end{tabular}
    \caption{Performance of conjugate gradient methods on Example \ref{Ex_5} associated with the cone $K_{1}$ given in \eqref{Cone_ex5_k1}}\label{Table_ex5_k1}
\end{table}

\vspace{-0.6cm}

\begin{table}[!h]
\centering
\begin{tabular}{p{1.1cm} p{1.1cm} p{1cm} p{1.1cm} p{1.1cm} p{1.1cm} p{1.1cm}}
\hline

\hline 
\multicolumn{1}{c}{\text{Method}}& \multicolumn{3}{c}{\text{Iteration counts}} & \multicolumn{3}{c}{\text{Time}} \\
\cmidrule(rl){2-4} \cmidrule(rl){5-7}
\vspace{0.1cm}
& min & mean & max & min & mean & max\\
\hline

\hline

DY & 0 & 0.15 & 3  & 0.1225 & 0.5538 & 1.7311  \\
PRP & 0 & 0.15 & 3  & 0.1225 & 0.5575 & 1.6302\\
HS & 0 & 0.15 & 3  & 0.1225 & 0.5600& 1.6680\\

\hline 

\hline
\end{tabular}
    \caption{Performance of the proposed conjugate gradient methods on Example \ref{Ex_5} associated with the cone $K_{2}$ given in \eqref{Cone_ex5_k2}}\label{Table_ex5_k2}
\end{table}

\vspace{-0.7cm}

\begin{figure}[h!]
\captionsetup[subfigure]{labelformat=empty}
\centering
\subfloat[]{
\includegraphics[width=0.49\textwidth]{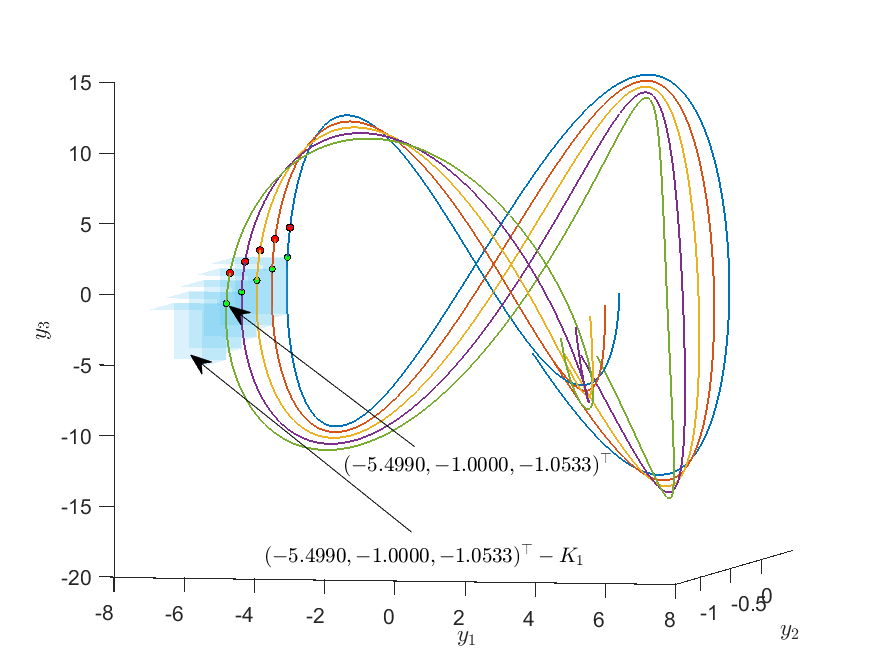}
\includegraphics[width=0.49\textwidth]{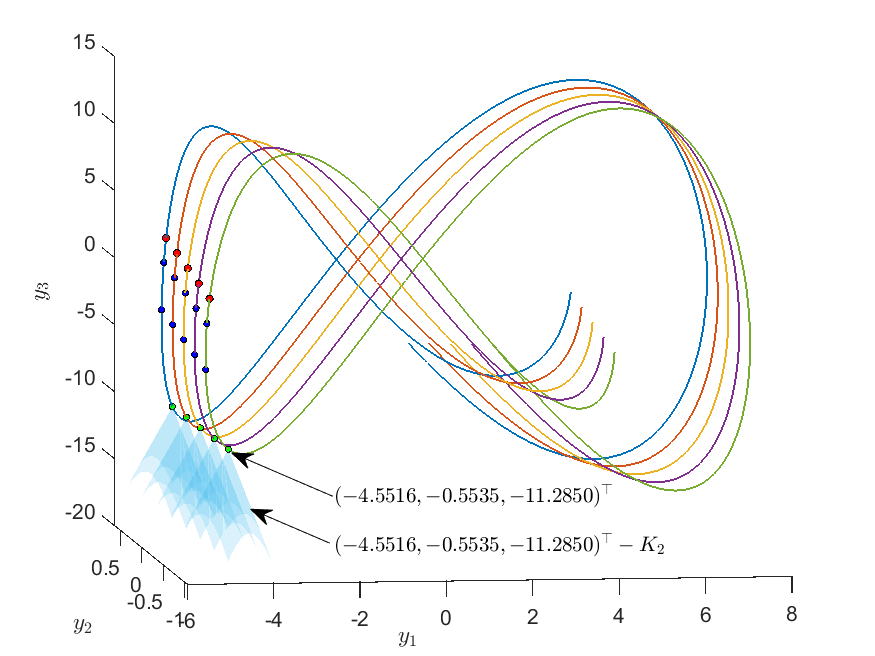}
}
\vspace{-0.6cm}
\caption{Sequence generated by Algorithm \ref{algo} with DY rule, in the image space of Example \ref{Ex_5}, for the initial point $x_0 = -10.9$ corresponding to $K_{1}$ and $K_{2}$ given in \eqref{Cone_ex5}} \label{fig_ex5}
\end{figure}

% With these experiments, we observe the following: 
% \begin{itemize}

%  \item Performance of the HS method is better than the other methods and is equivalent to the PRP method. However, it is important to recognize that these methods can perform similarly to others in certain problems. It should be noted that the FR method can perform better than the DY and the CD methods for some problems. 

% \item Even the set-valued mapping of the problem \eqref{SPL} and the initial point are the same, but the solution associated with two different cones may be different.

%\end{itemize}

\section{Conclusion}\label{conclusion} 
In this work, nonlinear conjugate gradient methods have been proposed to identify local weakly minimal points of the set optimization problem \eqref{SPL} under the lower set less preordering relation.  We have introduced the concept of $K$-descent direction and a sufficient descent condition for set-valued functions with the help of an auxiliary real-valued function. With the help of this auxiliary function, we have introduced the standard and strong Wolfe line searches for set-valued functions. Following this, we have established a result (Theorem \ref{exi.step} and Remark \ref{existence_of_std_wolfe}) on the existence of a step size for which the standard Wolfe or strong Wolfe conditions are satisfied. Then, we reported (in Theorem \ref{the. zoun.}) that for a given sequence $\{d_k\}$ of $K$-descent directions of a set-valued function, any sequence of iterates $\{x_{k}\}$ defined by $x_{k + 1} = x_k + \alpha_k d_k$ satisfies a Zoutendijk-like condition if the step size sequence  $\{\alpha_k\}$ satisfies the proposed standard Wolfe condition. Importantly, this step size existence result is derived without assuming that the ordering cone $K$ is finitely generated. Thereby, all the derived results in this study are applicable even for nonfinitely generated ordering cone $K$. \\ 

Thereafter, we have proposed a general scheme (Algorithm \ref{algo}) for nonlinear conjugate gradient methods for the set optimization problem. Then, we have reported (Proposition \ref{Pro_Bk_des_suf_des}) some choices of $\beta_k$-value so that the direction $d_k$ given in the general conjugate gradient Algorithm \ref{algo} is $K$-descent and sufficient descent. Subsequently, without explicitly restricting the parameter $\beta_{k}$, we have proved the global convergence (Theorem \ref{Conv_Algo}) of the proposed algorithm. \\

Next, we have proposed the DY, PRP, and HS rules to choose the $\beta_k$ value for set optimization problems. Then, in Theorem \ref{Con_DY} and Theorem \ref{theo req of prp con}, we have established the convergence of the proposed Algorithm \ref{algo} with DY, PRP, and HS rules. Finally, we have conducted numerical experiments to demonstrate the performance of the proposed methods. A comparison of the proposed DY, PRP and HS methods with the existing FR and CD methods for set optimization is also provided. \\ 

In the future, one may try to address the following issues. \\ 

\begin{itemize}
\item As mentioned in Remark \ref{Remark on Algo.}, a regular restart condition is used in Step 4 to make $d_{k}$ as a $K$-descent direction of $F$ at $x_{k}$. In the future, we may work on a new line search technique so that no regular restart condition is required.\\

% \item We have used a regular restart condition while defining the direction, which was further used in establishing a general convergence of Algorithm \ref{algo} and to ensure the direction $d_{k}$ will be $K$-descent at $x_{k}$ of $F$ for a specific choice of the parameter $\beta_{k}$.
% It is well-known that, in conjugate gradient methods, the direction $d_{k}$ is a linear combination of $u_{k}$ and a $K$-descent direction at $x_{k-1}$ of $F$. Accordingly, for a specific parameter $\beta_{k}$, to establish the direction $d_{k}:=u_{k}+\beta_{k}d_{k-1}$, where $d_{k-1}$ is $K$-descent direction at $x_{k-1}$ of $F$, is $K$-descent direction at $x_{k}$, commonly, we use a line search relation associated with the Jacobian of the objective function. 
% Notice that there is insufficient relation between the partition set $P_{k}$ and $P_{k-1}$, and the proposed line search condition is over a specific partition element $a_{k}\in P_{k}$. Therefore, it is difficult to find a relation between $\nabla f^{a_{k-1,\bar{j}}}(x_{k-1}),~\bar{j}\in \omega_{k-1}$, and $\nabla f^{a_{k,\tilde{j}}}(x_{k}),~\tilde{j}\in \omega_{k}$. There may be another way to introduce a regular restart condition on $d_{k}$. In the future, we may work on a new line search technique so that we will be able to remove the regular restart condition on the direction.\\ 

\item During the evaluation of the set $M_k:=\mathrm{Min}(F(x_k),K)$, in Step 1 of Algorithm \ref{algo}, we have used the crude way of pair-wise comparing the elements in $F(x_k)$. In the future, we may devise and employ an efficient way to find $M_k$. This may improve the efficiency of Algorithm \ref{algo}. \\

%     \item In this work, we completely skipped the notion of regularity, which is given in \cite{bouza2021steepest} as follows: A point $\bar{x}$ is called a regular point of \eqref{SPL} if the following conditions are satisfied: 
% \begin{enumerate}[label=(\roman*)]
%     \item $\mathrm{Min}(F(\bar{x}),K)=\mathrm{WMin}(F(\bar{x}),K)$, and 
%     \item there exists a neighbourhood of $\bar{x}$ in which the function $\omega$ is constant. \\ 
% \end{enumerate}
% However, to make an example of the set optimization problem \eqref{SPL} for which the stationary point does not hold regularity condition is not too difficult due to the first condition of regularity. But, with the regularity conditions at the stationary point, we may establish convergence with the $K$-descent direction.

\item Applying the proposed methods, future research may aim to find techniques to solve uncertain optimization problems with finite uncertainty. \\

\item No result on the rate of convergence of the proposed algorithm is derived in this study. Future studies can be done in this direction. \\

\item As the Gerstewitz scalarizing function requires a prespecified $e \in \text{int}(K)$, which may not be easy to choose, analysis of the proposed results with respect to other scalarizing functions can be done in the future. Also, in this paper, we have used the lower set less ordering of sets to derive results. Future work can use other preordering relations of sets to see if the derived results hold. \\ 

\end{itemize}

\subsection*{Acknowledgement}
Debdas Ghosh acknowledges the financial support of the research grants MATRICS
(MTR/2021/000696) and Core Research Grant (CRG/2022/001347) by the Science and Engineering Research Board, India. Ravi Raushan thankfully acknowledges financial support from CSIR, India, through a research fellowship (File No. 09/1217(13822)/2022-EMR-I) to carry out this research work.

\subsection*{Data availability}
There is no data associated with this paper.

%%%%%%%%%%%%%%%%%%%%%%

%\bibliographystyle{plain} % We choose the "plain" reference style
%\bibliography{sn-bibliography}
%%%%%%%%%%%%%%%%%%%%%%%%%%%

%\bibliographystyle{plain} 
%\bibliography{sn-Ref}% common bib file

\end{document}